\documentclass[journal,twoside,web]{ieeecolor}

\usepackage{generic}

\usepackage{cite}
\usepackage{amsmath,amssymb,amsfonts}
\usepackage{algorithmic}
\usepackage{graphicx}
\usepackage{algorithm,algorithmic}
\usepackage{hyperref}
\usepackage{graphicx}
\usepackage{multirow}
\usepackage{booktabs}     
\usepackage{tikz}
\usetikzlibrary{arrows.meta}
\hypersetup{hidelinks=true}
\usepackage{textcomp}

\usepackage{etoolbox}

\makeatletter
\def\@IEEEBIOskipN{1\baselineskip}

\patchcmd{\biography}
  {\vskip \@IEEEBIOskipN plus 1fil minus 0\baselineskip}
  {\vskip \@IEEEBIOskipN}
  {}
  {\GenericWarning{}{Biography spacing patch failed}}
\makeatother

\usepackage{amsthm}

\newtheorem{definition}{Definition}
\newtheorem{remark}{Remark}

\newtheorem{theorem}{Theorem}
\newtheorem{lemma}{Lemma}
\newtheorem{corollary}{Corollary}
\newtheorem{proposition}{Proposition}
\newtheorem{assumption}{Assumption}

\def\BibTeX{{\rm B\kern-.05em{\sc i\kern-.025em b}\kern-.08em
    T\kern-.1667em\lower.7ex\hbox{E}\kern-.125emX}}
\begin{document}
\title{A Feasible-Velocity Framework for Local Controllability of Nonlinear Systems with Zero-Excluding Input Constraints}
\author{Amal Bouazza$^{\ast}$, Mohamed Boutayeb, Mustapha Oudani
\thanks{This work was supported by the University of Lorraine (UL) and the International University of Rabat (UIR) under a joint Ph.D. program. The material in this paper will be partially presented 
at The 24th European Control Conference (ECC), July 7-10, 2026, Reykjav\'ik, Iceland}
\thanks{Amal Bouazza and Mohamed Boutayeb are with University of Lorraine, CNRS, CRAN, 54000 Nancy, France and with TICLab, International University of Rabat (UIR), 11100 Rabat, Morocco. (e-mails: \{amal.bouazza,mohamed.boutayeb\}\{@univ-lorraine.fr; @uir.ac.ma\})}
\thanks{Mustapha Oudani is with TICLab, International University of Rabat (UIR), 11100 Rabat, Morocco. (e-mails: mustapha.oudani@uir.ac.ma)}
\thanks{$^{\ast}$ Corresponding author.}}

\maketitle

\begin{abstract}
 This paper studies local controllability of nonlinear control-affine systems subject to state-dependent box constraints that strictly exclude the zero input. Such constraints arise naturally in cable-driven robots and other systems with strictly positive actuation, but fall outside classical small-time local controllability theory and existing frameworks for positive or cone-constrained controls. We introduce the admissible balancing set, an input-space object that classifies reference states without requiring the control distribution to have full rank. When an admissible balancing input lies in the interior of the input set, a locally uniform input shift recovers a symmetric-control system, allowing classical accessibility and small-time local controllability criteria to be applied. When no admissible balancing input exists, the feasible-velocity set is strictly separated from the origin. We show that the resulting separating covector defines a local barrier functional that increases at a uniform positive rate along every admissible trajectory, thereby providing a quantitative obstruction to small-time local controllability. This obstruction does not exclude finite-time reachability through trajectories leaving the barrier neighborhood, which motivates the notion of admissible excursions. The framework is illustrated on an underactuated planar cable-driven parallel robot, for which the barrier is certified numerically over a prescribed state neighborhood.
\end{abstract}


\section{Introduction}
\label{sec:introduction}
\IEEEPARstart{S}{mall}-time local controllability (STLC) is a central notion in
nonlinear control because it formalizes the ability to generate
arbitrarily small, locally reversible motions in arbitrarily short
time. For a control-affine system
\begin{equation}
	\dot{x}
	=
	f_0(x)+\sum_{i=1}^{m}u_i g_i(x),
	\qquad
	u\in\Omega(x),
	\label{eq:intro-system}
\end{equation}
the state is \(x\in\mathbb{R}^{n}\), where \(n\) is the state-space
dimension, and \(m\) is the number of scalar control inputs. Classical
geometric controllability results generally rely, after centering the
system at a controlled equilibrium, on an admissible control set that
contains the origin in its interior. This local symmetry permits
positive and negative control variations and underlies Lie-bracket
constructions, accessibility results, and sufficient conditions for
STLC developed by Sussmann, Krener, and Hermes \cite{sussmann1987general,krener1974generalization,hermes1982local}.

This interior-point assumption is violated in many systems with
unilateral or strictly positive actuation. Brammer \cite{brammer1972controllability}
established necessary and sufficient null-controllability conditions
for autonomous linear systems with positive controllers without
assuming that the origin is an interior point of the control set. His
analysis is based on the spectral structure of the linear drift and on
separation properties of time-integrated reachable sets. More
recently, Caillau, Dell'Elce, Herasimenka, and Pomet
\cite{caillau2025controllability,herasimenka2022controllability,herasimenka2023controllability} developed complementary controllability
conditions for nonlinear systems in which the convexified control set
contains the origin but need not contain a neighborhood of it. In the
cone-constrained solar-sail setting motivating those works, the origin
is a boundary point of the admissible force cone. Controllability is
then recovered through averaging or pushforwards along the drift flow,
with periodicity or fast oscillations playing a decisive role.

The present paper addresses a different, complementary geometry:
\begin{equation}
	0\notin\Omega(x).
	\label{eq:zero-excluding}
\end{equation}
Thus, the zero control is not admissible, even as a boundary value.
This distinction is structural. In the frameworks of
\cite{caillau2025controllability,herasimenka2022controllability,herasimenka2023controllability}, the zero-input drift motion provides the
reference evolution along which controlled directions can be
transported or averaged. Under \eqref{eq:zero-excluding}, that
reference motion is itself inadmissible. Likewise, Brammer's linear
theory \cite{brammer1972controllability} does not directly cover nonlinear dynamics,
state-dependent input boxes, or local controllability about a
reference state that need not initially be known to be an admissible
controlled equilibrium. The objective here is therefore not to
replace these theories, but to complete the geometric picture for
nonlinear systems whose admissible inputs remain strictly separated
from zero.

Strict zero exclusion arises naturally in cable-driven parallel
robots (CDPRs). A cable can pull but cannot push, and a vanishing
tension causes slackness and loss of actuation; consequently, each
cable tension must satisfy a strictly positive lower bound. Similar
constraints occur in systems driven by one-sided forces, pressure
differences, or actuators that cease to transmit effort at zero.
For a mechanical system with \(k\) configuration degrees of freedom,
the first-order state is typically \(x=(q,\dot q)\in\mathbb{R}^{n}\)
with \(n=2k\). The system is mechanically underactuated when the
number of independent actuators is smaller than \(k\), whereas
redundant or overactuation corresponds to having more actuator
channels than configuration degrees of freedom, subject to the rank
of the actuation map. This mechanical terminology should be
distinguished from the ambient state-space rank condition used below.

Let \(x_e\in\mathbb{R}^{n}\) denote the reference state about which
local controllability is investigated. No admissible equilibrium
assumption is imposed a priori at \(x_e\). Define
\begin{equation}
	G(x)
	:=
	\begin{bmatrix}
		g_1(x)&\cdots&g_m(x)
	\end{bmatrix}
\end{equation}
and the instantaneous feasible-velocity set
\begin{equation}
	V(x)
	:=
	\left\{
	f_0(x)+G(x)u:\;u\in\Omega(x)
	\right\}
	\subset\mathbb{R}^{n}.
	\label{eq:intro-feasible-velocity}
\end{equation}
The condition \(0\in\operatorname{int}V(x_e)\), where the interior is
taken in the ambient state space \(\mathbb{R}^{n}\), expresses a
favorable instantaneous geometry: the drift can be balanced and
feasible velocities surround the origin in every state-space
direction. However, this condition is meaningful only when the
control distribution has full row rank. Indeed, if
\(\operatorname{rank}G(x_e)<n\), then \(V(x_e)\) is contained in the
proper affine subspace
\[
f_0(x_e)+\operatorname{Im}G(x_e)
\]
and therefore has empty interior in \(\mathbb{R}^{n}\). In
particular, this issue affects underactuated systems, but it is not
limited to them: for a first-order representation of a mechanical
system, the control vector fields generally act only on the
acceleration block, so \(\operatorname{rank}G(x_e)<n=2k\) may hold
even when the mechanical system is fully or redundantly actuated.
Hence, actuator counting alone cannot determine whether
\(0\in\operatorname{int}V(x_e)\); the decisive quantity is the rank
of \(G(x_e)\) relative to the state dimension \(n\).

To obtain a classification that remains informative without any rank
assumption, we introduce the admissible balancing set
\begin{equation}
	\mathcal{B}(x_e)
	:=
	\left\{
	u\in\Omega(x_e):
	f_0(x_e)+G(x_e)u=0
	\right\}.
	\label{eq:intro-balancing-set}
\end{equation}
This input-space object separates three geometrically distinct
regimes. If
\(\mathcal{B}(x_e)\cap\operatorname{int}\Omega(x_e)\neq\varnothing\),
an interior admissible input balances the drift and a local input
shift restores symmetric control variations. If
\(\mathcal{B}(x_e)=\varnothing\), no admissible control balances the
drift and the origin lies outside \(V(x_e)\), producing a one-sided
first-order obstruction. The remaining boundary-balancing case,
\(\mathcal{B}(x_e)\neq\varnothing\) but
\(\mathcal{B}(x_e)\cap\operatorname{int}\Omega(x_e)=\varnothing\),
leads after shifting to a tangent-cone constraint and requires a
separate asymmetric analysis. When
\(\operatorname{rank}G(x_e)=n\), this classification reduces exactly
to the position of the origin in the interior, exterior, or boundary
of \(V(x_e)\). When \(\operatorname{rank}G(x_e)<n\), it remains
operational although \(\operatorname{int}V(x_e)=\varnothing\).

The contributions of this paper are fourfold. First, we introduce the admissible balancing set
\(\mathcal B(x_e)\) and use it to classify reference states under
state-dependent zero-excluding input constraints. The classification
requires no rank assumption and coincides with the interior,
boundary, and exterior feasible-velocity regimes whenever
\(\operatorname{rank}G(x_e)=n\).

Second, in the interior-balancing regime, we prove that continuity of
the input bounds makes a constant balancing-input shift uniformly
admissible on a state neighborhood. This reduction transfers
Krener's accessibility result and Sussmann's
\(\mathcal S(\theta)\)-criterion to the original constrained system.

Third, in the no-balancing regime, we establish a quantitative
first-order obstruction to STLC. Projection of the origin onto
\(V(x_e)\), followed by local uniformization, produces a covector
\(\lambda\) and a rate \(\alpha'>0\) such that
\[
\langle\lambda,x(t)-x_e\rangle\geq\alpha't
\]
for every admissible trajectory that remains near \(x_e\).

Fourth, we specialize the framework to a planar underactuated CDPR.
The balancing-set classification partitions its static operating
domain, while an interval branch-and-bound computation certifies the
no-balancing barrier on a continuous neighborhood of a representative
configuration. Extremal, adversarial, and randomly switched
trajectory simulations illustrate the certificate.

\section{Related Work}
\label{sec:related-work}

Classical nonlinear controllability theory characterizes accessibility
and STLC through Lie-algebraic constructions and signed local control
variations \cite{sussmann1972controllability,hermann1977nonlinear,
krener1974generalization,hermes1982local,sussmann1987general}.
Subsequent obstruction analyses
\cite{jafarpour2020small,gherdaoui2025quadratic,
beauchard2026unified} retain this local symmetric-control setting.
The present work instead considers state-dependent input boxes that
strictly exclude zero and reference states that need not initially be
known to be controlled equilibria.

For one-sided inputs, Brammer established a complete
null-controllability characterization for autonomous linear systems
with positive controllers \cite{brammer1972controllability}; related
linear developments include
\cite{saperstone1971controllability,saperstone1973global,
heymann1975controllability,pachter1978control}. Related work includes
constrained STLC for linear systems \cite{krastanov2008constrained}
and output regulation under actuator saturation by Lin, Stoorvogel,
and Saberi \cite{lin1996output}. These saturation settings typically
retain zero as an admissible input, unlike the present framework.
Nonlinear constrained controllability has also been studied for
unilateral or cone-valued inputs
\cite{goodwine1996controllability,klamka1996constrained}.
Most directly related are the results of Caillau, Dell'Elce,
Herasimenka, and Pomet
\cite{caillau2025controllability,herasimenka2022controllability,
herasimenka2023controllability}, where the convexified control set
contains the origin as a boundary point and controllability is
recovered through drift-flow pushforwards or averaging. Here zero
itself is inadmissible, so the drift trajectory cannot serve as an
admissible reference motion. Our balancing-set classification and
feasible-velocity separation therefore address a complementary
geometry.

For infinite-dimensional systems, positivity-constrained
controllability of fractional parabolic equations can exhibit a
strictly positive minimal time \cite{biccari2020fractional}.
Extending feasible-velocity separation to such systems is an open
direction.

Cable-driven parallel robots provide a representative application
because cable tensions are strictly positive and underactuated
configurations have rank-deficient control distributions. Existing
work has primarily addressed stability, trajectory planning, tension
distribution, and tracking
\cite{carricato2013stability,ida2019,gouttefarde2015versatile,
pott2013cable,bettega2023model}. A systematic local controllability
analysis combining strict-positive tensions, state-dependent bounds,
and mechanical underactuation remains limited. The present framework
fills this gap and uses standard tools from set-valued and convex
analysis \cite{aubin1990set,rockafellar1998variational,
rockafellar1997convex} to turn pointwise feasible-velocity separation
into a uniform trajectory barrier.

\section{System Definition}
\label{sec:system-definition}

Let $\nu \in \mathbb{Z}_{\geq 1} \cup \{\infty,\omega\}$ denote a
regularity class. We consider a $C^\nu$ control-affine system
defined by the family
$$    \mathcal{F}:=\{f_0,g_1,\ldots,g_m\}
$$
of vector fields on $\mathbb{R}^n$, with dynamics
\begin{equation}
	\label{eq:system_dynamique}
	\dot x(t)
	=
	f_0(x(t))+\sum_{i=1}^m u_i(t)g_i(x(t)),
	\quad
	u(t)\in \Omega(x(t)),
\end{equation}
where
$$  \Omega(x):=
\prod_{i=1}^m
[\underline u_i(x),\overline u_i(x)]
\subset \mathbb{R}^m $$
is a state-dependent box.

\begin{assumption}
	\label{ass:1}
	The bound functions
	$\underline u_i,\overline u_i:\mathbb{R}^n\to\mathbb{R}$
	are continuous and satisfy
	\[
	\underline u_i(x)<\overline u_i(x),
	\qquad
	x\in\mathbb{R}^n,\quad i=1,\ldots,m.
	\]
\end{assumption}

The central feature of the framework is the strict exclusion of the
zero control. We assume either
\[
0<\underline u_i(x)<\overline u_i(x)
\qquad
\text{for all }x,i,
\]
or
\[
\underline u_i(x)<\overline u_i(x)<0
\qquad
\text{for all }x,i.
\]
Thus \(0\notin\Omega(x)\) for every \(x\). This models actuators that
lose their functional role at zero input, such as cables that become
slack when their tension vanishes. Zero exclusion is not by itself an obstruction to controllability: a
nonzero admissible input may balance the drift; the obstruction appears
only when the induced feasible velocity set has a one-sided geometry,
as made precise below. This pointwise zero exclusion is uniform on
compact subsets of the state space, as a direct consequence of the
continuity of the input bounds.

\begin{lemma}[Compact-uniform zero exclusion]
	\label{lem:compact-zero-exclusion}
	Under Assumption~\ref{ass:1} and either sign-definite input-domain
	condition, for every compact set $K\subset\mathbb{R}^n$ there exists
	$\delta_K>0$ such that
	\[
	\operatorname{dist}(0,\Omega(x))\geq \delta_K,
	\qquad x\in K.
	\]
\end{lemma}

\begin{proof}
	In the positive case, define
	\(\psi_i(x):=\underline u_i(x)\). In the negative case, define
	\(\psi_i(x):=-\overline u_i(x)\). In both cases, \(\psi_i\) is
	continuous and strictly positive on \(K\). Hence, by compactness,
	\[
	\delta_i:=\min_{x\in K}\psi_i(x)>0.
	\]
	Set \(\delta_K:=\min_i\delta_i>0\). For every \(x\in K\) and every
	\(u\in\Omega(x)\), one has \(|u_i|\geq \psi_i(x)\geq \delta_K\) for
	each \(i\). Therefore \(\|u\|\geq \delta_K\), and the conclusion
	follows.
\end{proof}

We now define admissible controls and trajectories under the
state-dependent constraint. For \(T>0\) and a trajectory
\(x(\cdot)\) defined on \([0,T]\), the set of admissible controls
along \(x(\cdot)\) is
\begin{align}
	\mathcal{U}_{\mathrm{ad}}\big(x(\cdot);0,T\big)
	:=
	\Big\{
	u\in L^\infty(0,T;&\mathbb{R}^m):
	\ u(t)\in\Omega(x(t))
	\nonumber\\
	&\text{for a.e. }t\in(0,T)
	\Big\}.
	\label{eq:admissible-controls-general}
\end{align}
A trajectory is an absolutely continuous map
\(x:[0,T]\to\mathbb{R}^n\) satisfying
\eqref{eq:system_dynamique} almost everywhere with
\(u\in\mathcal{U}_{\mathrm{ad}}(x(\cdot);0,T)\). Existence of such
trajectories is standard through the differential-inclusion
interpretation of~\eqref{eq:system_dynamique}
\cite{filippov2013differential,aubin2008differential}.
\subsection{Feasible Velocity Set and Balancing Inputs}

For a given state \(x\in\mathbb{R}^n\), define
$$G(x):=[g_1(x)\ \cdots\ g_m(x)]\in\mathbb{R}^{n\times m}. $$
The \emph{feasible velocity set} at \(x\) is
\begin{equation}
	\label{eq:velocity_set}
	V(x)
	:=
	\left\{
	f_0(x)+\sum_{i=1}^m u_i g_i(x):
	u\in\Omega(x)
	\right\}.
\end{equation}
Equivalently,
$$V(x)=f_0(x)+G(x)\Omega(x).$$
Thus \(V(x)\) is the set of all instantaneous velocities that can be
generated at \(x\) by admissible inputs. The position of the origin
relative to \(V(x)\) determines whether the state can be held fixed by
some admissible input and whether a first-order one-sided obstruction
is present.

Throughout the paper, \(x_e\in\mathbb{R}^n\) denotes a reference
state at which local controllability properties are analyzed. This
reference state is not assumed a priori to be an admissible controlled
equilibrium. It may fail to be equilibrable, or it may be equilibrable
only by an input that is not admissible under the constraint
\(\Omega(x_e)\). We therefore distinguish between algebraic balancing
inputs, which solve the balance equation in \(\mathbb{R}^m\), and
admissible balancing inputs, which also satisfy the input constraint.

\begin{definition}[Controlled equilibrium]
	\label{def:controlled-equilibrium}
	A state \(x_e\in\mathbb{R}^n\) is a \emph{controlled equilibrium} of
	\((\mathcal{F},\Omega)\) if there exists \(u_e\in\Omega(x_e)\) such
	that
	\begin{equation}
		f_0(x_e)+\sum_{i=1}^m u_{e,i}g_i(x_e)=0.
	\end{equation}
	Such a \(u_e\) is called a \emph{balancing input} at \(x_e\).
\end{definition}

\begin{definition}[Balancing inputs and admissible balancing set]
	\label{def:balancing-set}
	For a reference state \(x_e\in\mathbb{R}^n\), define:
	\begin{itemize}
		\item the \emph{unconstrained balancing set}
		\begin{equation}
			\mathcal{B}_e(x_e):=
			\left\{
			u\in\mathbb{R}^m:
			f_0(x_e)+\sum_{i=1}^m u_i g_i(x_e)=0
			\right\};
		\end{equation}
		\item the \emph{admissible balancing set}
		\begin{equation}
			\mathcal{B}(x_e)
			:=
			\mathcal{B}_e(x_e)\cap\Omega(x_e).
		\end{equation}
	\end{itemize}
\end{definition}

The pair \((\mathcal{B}_e(x_e),\mathcal{B}(x_e))\) refines the
classification of \(x_e\) into physically meaningful categories. If
\(\mathcal{B}(x_e)\neq\emptyset\), then \(x_e\) is an admissible
controlled equilibrium: the system can be held at \(x_e\) by an
admissible constant input. If
\(\mathcal{B}_e(x_e)\neq\emptyset\) but
\(\mathcal{B}(x_e)=\emptyset\), then the balance equation is solvable
in \(\mathbb{R}^m\) but not inside the admissible input set
\(\Omega(x_e)\). In this case, \(x_e\) is a \emph{virtual balance
	point}: it is mathematically equilibrable but physically
unrealizable, for instance a CDPR configuration requiring negative
cable tensions. Finally, if \(\mathcal{B}_e(x_e)=\emptyset\), no input
in \(\mathbb{R}^m\) can balance the dynamics; equivalently,
\begin{equation}
	f_0(x_e)\notin \operatorname{Im}G(x_e).
\end{equation}

The controllability analysis in Sections~\ref{sec:case-A} and
\ref{sec:case-B} focuses on two regimes treated in the present paper.
The first is the interior-balancing regime
\begin{equation}\label{eq:capnonempty}
	\mathcal{B}(x_e)\cap\operatorname{int}\Omega(x_e)\neq\emptyset,
\end{equation}
where an admissible input shift is available. The second is the
no-balancing regime
\begin{equation}
	\mathcal{B}(x_e)=\emptyset,
\end{equation}
where no admissible input can balance the drift. The finer distinction
provided by \(\mathcal{B}_e(x_e)\) becomes relevant in applications
(Section~\ref{sec:application}). By definition,
\begin{equation}
	\mathcal{B}(x_e)\neq\emptyset
	\quad\Longleftrightarrow\quad
	0\in V(x_e),
\end{equation}
which links the input-space classification to the geometry of the
feasible velocity set. In the linear specialization \(x_e=0\), with
\(f_0(x)=Ax\) and \(g_i(x)=b_i\), this condition recovers Brammer's
hypothesis
\[
\Omega\cap\ker B_{\rm lin}\neq\emptyset
\]
for the linear system
\[
\dot x=Ax+B_{\rm lin}u,
\]
where \(B_{\rm lin}:=[b_1\ \cdots\ b_m]\)
\cite[(1.3)]{brammer1972controllability}. The notation
\(B_{\rm lin}\) is used here only to avoid confusion with the
admissible balancing set \(\mathcal{B}(x_e)\).

The controllability analysis relies on two structural properties of
\(V(\cdot)\): pointwise compactness and convexity, used for separating
hyperplanes when \(0\notin V(x_e)\) in Section~\ref{sec:case-B}, and
Hausdorff continuity, used for local uniformization in
Sections~\ref{sec:case-A} and~\ref{sec:case-B}.

\begin{lemma}[Properties of the feasible velocity set]
	\label{lem:V-properties}
	Under Assumption~\ref{ass:1}, the feasible velocity set \(V(\cdot)\)
	defined in~\eqref{eq:velocity_set} satisfies:
	\begin{itemize}
		\item[$i)$] for every \(x\in\mathbb{R}^n\), \(V(x)\) is a compact
		convex subset of \(\mathbb{R}^n\);
		\item[$ii)$] the set-valued map \(x\mapsto V(x)\) is continuous in
		the Hausdorff sense on \(\mathbb{R}^n\).
	\end{itemize}
\end{lemma}

\begin{proof}
	\(i)\) For fixed \(x\), the map
	\[
	\phi_x:u\mapsto f_0(x)+\sum_{i=1}^m u_i g_i(x)
	\]
	is affine in \(u\). The input set \(\Omega(x)\) is a Cartesian product
	of compact intervals and is therefore compact and convex. Since the
	image of a compact convex set under a continuous affine map is compact
	and convex,
	\[
	V(x)=\phi_x(\Omega(x))
	\]
	is compact and convex.
	
	\smallskip
	\(ii)\) We prove Hausdorff continuity in two steps.
	
	First, \(\Omega(\cdot)\) is Hausdorff continuous. Indeed, for
	\(x,x_0\in\mathbb{R}^n\), the sets \(\Omega(x)\) and \(\Omega(x_0)\)
	are boxes whose endpoints depend continuously on \(x\). A
	coordinate-wise estimate gives
	\begin{align*}
		d_H\big(\Omega(x),\Omega(x_0)\big)
		\leq
		\sum_{i=1}^m
		\big|\underline u_i(x)-\underline u_i(x_0)\big|
		+
		\big|\overline u_i(x)-\overline u_i(x_0)\big|,
	\end{align*}
	where \(d_H\) denotes the standard Hausdorff distance. The
	right-hand side tends to zero as
	\(x\to x_0\), by continuity of the bound functions.
	
	Second, the map
	\[
	\phi(x,u):=f_0(x)+\sum_{i=1}^m u_i g_i(x)
	\]
	is jointly continuous in \((x,u)\). Since \(\Omega(\cdot)\) is
	Hausdorff continuous with non-empty compact values, the standard
	continuity theorem for compact-valued images implies that
	\[
	V(x)=\phi(\{x\}\times\Omega(x))
	\]
	is Hausdorff continuous with non-empty compact values.
\end{proof}

\subsection{Reachability, Accessibility, and STLC}

For \(x_e\in\mathbb{R}^n\), \(T>0\), and an open neighborhood
\(\mathcal{N}_{x_e}\) of \(x_e\), the constrained reachable set inside
\(\mathcal{N}_{x_e}\) is defined by
\begin{align}
	\mathcal{R}_{\mathcal{F}}^\Omega&(\leq T,x_e;\mathcal{N}_{x_e})
	:=
	\{x_e\}
	\cup
	\Big\{
	x(\tau):
	\ 0<\tau\leq T,
	\nonumber\\
	&\ x:[0,\tau]\to\mathcal{N}_{x_e}
	\text{ admissible},
	\quad x(0)=x_e
	\Big\}.
	\label{eq:local-reachable-set}
\end{align}
The point \(x_e\) is included explicitly to account for reachability
at time zero. Under zero-excluding constraints, \(x_e\) need not be
reachable again at any positive time, and forward and backward
reachability are generally not symmetric.

We write
\[
\mathcal{R}_{\mathcal{F}}^\Omega(\leq T,x_e)
:=
\mathcal{R}_{\mathcal{F}}^\Omega(\leq T,x_e;\mathbb{R}^n).
\]

We recall the standard notions of local accessibility, LARC, and STLC
in the constrained setting. These notions are classical in geometric
nonlinear control~\cite{hermann1977nonlinear,krener1974generalization,
	sussmann1987general,coron2007control}; here, reachability is always
understood with respect to the admissible state-dependent constraint
\(\Omega(x)\).

\begin{definition}[Local accessibility under \(\Omega\)]
	\label{def:local-accessibility}
	The system \((\mathcal{F},\Omega)\) is \emph{locally accessible from
		\(x_e\) under \(\Omega\)} if there exists an open neighborhood
	\(\mathcal{N}_{x_e}\) of \(x_e\) such that
	\[
	\mathcal{R}_{\mathcal{F}}^\Omega(\leq T,x_e;\mathcal{N}_{x_e})
	\]
	has non-empty interior in \(\mathbb{R}^n\) for every \(T>0\).
\end{definition}

For a family
\(\mathcal{F}=\{f_0,g_1,\ldots,g_m\}\) of sufficiently smooth vector
fields on \(\mathbb{R}^n\), the Lie algebra generated by
\(\mathcal{F}\), denoted \(\mathcal{L}(\mathcal{F})\), is the smallest
Lie subalgebra of the space of vector fields that contains
\(\mathcal{F}\) and is closed under the Lie bracket
\([\cdot,\cdot]\). Equivalently, \(\mathcal{L}(\mathcal{F})\) contains
the original vector fields \(f_0,g_1,\ldots,g_m\) together with all
iterated Lie brackets that can be formed from them. For
\(x_e\in\mathbb{R}^n\), the evaluation
\(\mathcal{L}(\mathcal{F})(x_e)\subseteq\mathbb{R}^n\) is the linear
subspace generated by the values at \(x_e\) of all vector fields in
\(\mathcal{L}(\mathcal{F})\).

\begin{definition}[Lie Algebra Rank Condition]
	\label{def:larc}
	The system \((\mathcal{F},\Omega)\) satisfies the
	\emph{Lie Algebra Rank Condition (LARC)} at \(x_e\) if
	\[
	\dim\mathcal{L}(\mathcal{F})(x_e)=n.
	\]
\end{definition}

LARC is the classical infinitesimal criterion for local accessibility
when the admissible control set contains a symmetric neighborhood of
the origin~\cite{krener1974generalization,hermann1977nonlinear}. In
the present zero-excluding setting, this hypothesis is not available
directly. When the interior-balancing condition
\eqref{eq:capnonempty} holds, one can perform the shift
\(u=u_e+\eta\), which transforms the system locally into one with a
symmetric control set for the shifted input \(\eta\). The shift itself
is classical; the point here is to identify, through the admissible
balancing set, when this reduction is valid under state-dependent
zero-excluding constraints. This reduction is carried out in
Section~\ref{sec:case-A}.

The Lie-algebraic nature of LARC makes it the appropriate
accessibility criterion for underactuated systems \((m<n)\). Although
the direct control directions \(g_1(x_e),\ldots,g_m(x_e)\) cannot span
\(\mathbb{R}^n\) by themselves when \(m<n\), iterated Lie brackets
involving \(f_0,g_1,\ldots,g_m\) may generate the missing directions.
Thus no criterion based only on \(\operatorname{rank}G(x_e)\) can
capture local accessibility in such systems.

\begin{definition}[STLC under \(\Omega\)]
	\label{def:stlc}
	The system \((\mathcal{F},\Omega)\) is \emph{small-time locally
		controllable} (STLC) at \(x_e\) if there exists an open neighborhood
	\(\mathcal{N}_{x_e}\) of \(x_e\) such that, for every
	\(\varepsilon>0\), there exists \(r(\varepsilon)>0\) satisfying
	\[
	B_{r(\varepsilon)}(x_e)
	\subseteq
	\mathcal{R}_{\mathcal{F}}^\Omega
	(\leq \varepsilon,x_e;\mathcal{N}_{x_e}).
	\]
\end{definition}
The following first-order condition records the absence of a
linear separating covector for the local reachable set. It will be
useful for interpreting the obstruction obtained in the no-balancing
regime.

\begin{definition}[Convexified first-order STLC condition]
	\label{def:first-order-stlc}
	The system \((\mathcal{F},\Omega)\) satisfies the
	\emph{convexified first-order STLC condition} at \(x_e\) if, for every
	open neighborhood \(\mathcal{N}_{x_e}\) of \(x_e\) and every \(T>0\),
	\[
	\operatorname{co}\,
	T_{\mathcal{R}_{\mathcal{F}}^\Omega
		(\leq T,x_e;\mathcal{N}_{x_e})}(x_e)
	=
	\mathbb{R}^n,
	\]
	where \(T_S(x_e)\) denotes the Bouligand tangent cone to a set
	\(S\ni x_e\):
	\[
	T_S(x_e)
	:=
	\left\{
	v\in\mathbb{R}^n:
	\exists\, t_k\to 0^+,\ \exists\, x_k\in S,
	\frac{x_k-x_e}{t_k}\to v
	\right\}.
	\]
	Equivalently, for every such \(\mathcal{N}_{x_e}\) and \(T\), there is
	no nonzero covector \(\lambda\) such that
	\[
	\langle \lambda,v\rangle\geq 0
	\qquad
	\forall v\in
	T_{\mathcal{R}_{\mathcal{F}}^\Omega
		(\leq T,x_e;\mathcal{N}_{x_e})}(x_e).
	\]
\end{definition}
\begin{remark}[Relation with STLC]
	\label{rem:stlc-first-order}
	STLC implies the convexified first-order STLC condition, but the
	converse need not hold. The condition only excludes first-order
	one-sided obstructions. Higher-order obstructions may still prevent
	STLC, which is why additional bracket conditions, such as Sussmann's
	\(\mathcal{S}(\theta)\)-condition, are needed in the
	interior-balancing regime.
\end{remark}
\begin{remark}[Dual interpretation and role in Case~(B)]
	\label{rem:fostlc-dual}
	The preceding condition is the dual formulation of the absence of a
	first-order separating covector. If
	\[
	\operatorname{co}\,
	T_{\mathcal{R}_{\mathcal{F}}^\Omega
		(\leq T,x_e;\mathcal{N}_{x_e})}(x_e)
	\neq
	\mathbb{R}^n,
	\]
	then the supporting-hyperplane theorem yields a nonzero covector
	\(\lambda\) such that $\langle \lambda,v\rangle\geq 0$ for every first-order reachable direction \(v\).
	
	The closed convex hull appears because the existence of a separating
	covector is unchanged when the tangent cone is replaced by its closed
	convex hull. The novelty in the no-balancing
	regime is not the separation theorem itself, but the identification of
	the feasible velocity set \(V(x_e)\) as the correct object to separate
	from the origin, and the conversion of this separation into a local
	barrier proving failure of STLC.
\end{remark}

STLC implies local accessibility but is strictly stronger. A system may
be locally accessible without being STLC when forward reachability is
biased in a particular direction. In the present setting,
zero-excluding constraints can preclude STLC while still leaving open
the possibility of finite-time reachability through admissible
excursions, as discussed in Section~\ref{sec:case-B}.

\section{Geometric Classification of Reference States}
\label{sec:classification}

Before introducing the input-space classification, it is useful to
recall the natural ambient-space viewpoint based on the feasible
velocity set \(V(x_e)\). If
\[
0\in\operatorname{int}V(x_e),
\]
then the drift can be balanced by an admissible input and, at the
instantaneous level, admissible velocities are available in every
direction of \(\mathbb{R}^n\). This condition reflects the local
signed-variation geometry underlying classical sufficient conditions
for local accessibility and STLC
\cite{krener1974generalization,sussmann1987general}. In full-row-rank
situations, \(\operatorname{rank}G(x_e)=n\), it provides a natural
ambient-space way to express the favorable regime.

However, this ambient-space viewpoint becomes structurally
inapplicable as soon as \(\operatorname{rank}G(x_e)<n\), since
\(V(x_e)\) then lies in a proper affine subspace of \(\mathbb{R}^n\)
and has empty interior
(Proposition~\ref{prop:empty-interior-underactuated}). This is not an
exotic situation: it includes all underactuated systems \((m<n)\) and,
more generally, all control-affine systems written in second-order
state form \((q,\dot q)\in\mathbb{R}^{2k}\), where the control
distribution acts only on the velocity block and has at most \(k\)
independent columns in \(\mathbb{R}^{2k}\). Mechanical systems,
fully actuated or not, thus fall outside this ambient-space interior
condition when analyzed at the state level.\footnote{One could attempt
	to relax \(0\in\operatorname{int}V(x_e)\) to
	\(0\in\operatorname{relint}V(x_e)\) in the affine hull of \(V(x_e)\).
	This formulation, however, would require restating the accessibility
	and STLC theorems of~\cite{krener1974generalization,sussmann1987general}
	in the relative topology of a state-dependent affine subspace, an
	unwieldy reformulation. The \(\mathcal{B}\)-classification adopted
	here avoids the issue by operating directly on the input set in
	\(\mathbb{R}^m\), where the input box has non-empty interior
	independently of the rank of \(G(x_e)\).}

The point of the present section is to formalize a different
viewpoint: rather than classifying reference states by the position of
the origin relative to the ambient feasible velocity set \(V(x_e)\),
we classify them by the existence and location of balancing inputs
inside the admissible input set \(\Omega(x_e)\). This leads to the
admissible balancing set
\(\mathcal{B}(x_e)\subseteq\Omega(x_e)\subset\mathbb{R}^m\)
(Definition~\ref{def:balancing-set}), which lives in input space
rather than state space. To the best of our knowledge, this
input-space classification has not been used as a systematic framework
for nonlinear systems with state-dependent zero-excluding input
constraints.

The \(\mathcal{B}\)-classification developed below yields a single
geometric framework that:
\begin{itemize}
	\item[(i)] covers all control-affine systems with zero-excluding
	inputs under no rank assumption on \(G(x_e)\);
	\item[(ii)] coincides with the ambient-space condition
	\(0\in\operatorname{int}V(x_e)\) when the latter applies, namely when
	\(\operatorname{rank}G(x_e)=n\)
	(Proposition~\ref{prop:coincidence-fully-actuated});
	\item[(iii)] remains operational in the structurally underactuated
	regime \(\operatorname{rank}G(x_e)<n\) that motivates the
	cable-driven robotic application of Section~\ref{sec:application}
	(Proposition~\ref{prop:strict-generality-underactuated}).
\end{itemize}

\subsection{The \(\mathcal{B}\)-based classification}
\label{subsec:B-classification}

We now classify reference states according to the existence and
location of admissible balancing inputs.

The three regimes are:
\[
\begin{array}{ll}
	\text{(A)} &
	\mathcal{B}(x_e)\cap\operatorname{int}\Omega(x_e)\neq\emptyset,
	\\[1mm]
	\text{(B)} &
	\mathcal{B}(x_e)=\emptyset,
	\\[1mm]
	\text{(C)} &
	\mathcal{B}(x_e)\neq\emptyset
	\text{ and }
	\mathcal{B}(x_e)\cap\operatorname{int}\Omega(x_e)=\emptyset.
\end{array}
\]
Case~(A) admits an interior admissible balancing input, allowing
locally signed control variations after an input shift. Case~(B)
admits no admissible balancing input and leads to the one-sided
obstruction analyzed in Section~\ref{sec:case-B}. Case~(C) admits
balancing only on \(\partial\Omega(x_e)\). These three regimes form a
partition of \(\mathbb{R}^n\): they are mutually exclusive and
exhaustive, so every \(x_e\in\mathbb{R}^n\) falls in exactly one case.

\begin{remark}[Boundary-balancing regime]
	\label{rem:boundary-balancing}
	The boundary-balancing regime
	\[
	\mathcal{B}(x_e)\neq\emptyset,\qquad
	\mathcal{B}(x_e)\cap\operatorname{int}\Omega(x_e)=\emptyset
	\]
	is identified by the \(\mathcal{B}\)-classification but is not pursued
	in the present paper. In this regime, every admissible balancing input
	lies on the boundary of the input box, so that an input shift produces
	a one-sided tangent-cone control set rather than a symmetric
	neighborhood of the origin. This regime requires a separate analysis
	of asymmetric polyhedral control cones and is left for future work.
\end{remark}

\begin{remark}[Equivalent formulation via \(\mathcal{B}_e\)]
	\label{rem:equivalent-formulation}
	Since \(\operatorname{int}\Omega(x_e)\subseteq\Omega(x_e)\), one has
	\[
	\mathcal{B}(x_e)\cap\operatorname{int}\Omega(x_e)
	=
	\mathcal{B}_e(x_e)\cap\operatorname{int}\Omega(x_e),
	\]
	so Case~(A) is equivalently formulated as
	\[
	\mathcal{B}_e(x_e)\cap\operatorname{int}\Omega(x_e)\neq\emptyset.
	\]
	
	This equivalent form is convenient when \(\mathcal{B}_e(x_e)\) admits
	a closed-form expression, for instance the singleton given by
	Cramer's rule for the planar CDPR of Section~\ref{sec:application}.
	We retain the formulation in terms of \(\mathcal{B}(x_e)\) throughout:
	it treats Cases~(A), (B), and (C) uniformly and avoids confusing
	Case~(B) with virtual balance points, which satisfy
	\(\mathcal{B}_e(x_e)\neq\emptyset\) but
	\(\mathcal{B}(x_e)=\emptyset\).
\end{remark}
\subsection{Relation to the ambient condition
	\(0\in\operatorname{int}V(x_e)\)}
\label{subsec:relation-to-V}

We now make precise the relation between the
\(\mathcal{B}\)-classification and the ambient-space condition on
\(V(x_e)\), distinguishing the two structural regimes
\(\operatorname{rank}G(x_e)=n\) and
\(\operatorname{rank}G(x_e)<n\).

\begin{proposition}[Empty interior of \(V(x_e)\) in the rank-deficient regime]
	\label{prop:empty-interior-underactuated}
	Let \(x_e\in\mathbb{R}^n\) and
	\(G(x_e):=[g_1(x_e),\ldots,g_m(x_e)]\in\mathbb{R}^{n\times m}\). If
	\(\operatorname{rank}G(x_e)<n\), then
	\[
	\operatorname{int}V(x_e)=\emptyset.
	\]
	In particular, this holds whenever \(m<n\).
\end{proposition}

\begin{proof}
	The feasible velocity set satisfies
	\[
	V(x_e)\subseteq f_0(x_e)+\operatorname{Im}G(x_e).
	\]
	Hence \(V(x_e)\) is contained in an affine subspace of
	\(\mathbb{R}^n\) of dimension \(\operatorname{rank}G(x_e)<n\). Such an
	affine subspace has empty interior in \(\mathbb{R}^n\), and therefore
	so does \(V(x_e)\).
\end{proof}

\begin{proposition}[Coincidence with \(V\) in the full-row-rank regime]
	\label{prop:coincidence-fully-actuated}
	Assume \(\operatorname{rank}G(x_e)=n\). Then
	\begin{equation}
		\label{eq:coincidence}
		0\in\operatorname{int}V(x_e)
		\quad\Longleftrightarrow\quad
		\mathcal{B}(x_e)\cap\operatorname{int}\Omega(x_e)\neq\emptyset,
	\end{equation}
	and the boundary and exterior regimes coincide pointwise:
	\(0\in\partial V(x_e)\) if and only if Case~(C) holds, and
	\(0\notin V(x_e)\) if and only if Case~(B) holds.
\end{proposition}

\begin{proof}
	Since \(\operatorname{rank}G(x_e)=n\), the affine map
	\[
	\phi_{x_e}:u\mapsto f_0(x_e)+G(x_e)u
	\]
	is an open map from \(\mathbb{R}^m\) onto \(\mathbb{R}^n\). Since
	\(\Omega(x_e)\) is a box with non-empty interior, it follows that
	\[
	\phi_{x_e}\big(\operatorname{int}\Omega(x_e)\big)
	=
	\operatorname{int}V(x_e).
	\]
	Hence \(0\in\operatorname{int}V(x_e)\) if and only if there exists
	\(u_e\in\operatorname{int}\Omega(x_e)\) such that
	\(\phi_{x_e}(u_e)=0\), which is precisely
	\[
	\mathcal{B}(x_e)\cap\operatorname{int}\Omega(x_e)\neq\emptyset.
	\]
	
	Moreover, \(0\notin V(x_e)\) if and only if
	\(\mathcal{B}(x_e)=\emptyset\). The remaining case,
	\(0\in V(x_e)\) but \(0\notin\operatorname{int}V(x_e)\), is equivalent
	to
	\[
	\mathcal{B}(x_e)\neq\emptyset
	\qquad\text{and}\qquad
	\mathcal{B}(x_e)\cap\operatorname{int}\Omega(x_e)=\emptyset,
	\]
	which is exactly Case~(C). This proves the claimed coincidence of the
	interior, exterior, and boundary regimes.
\end{proof}

\begin{proposition}[Strict generality in the rank-deficient regime]
	\label{prop:strict-generality-underactuated}
	Assume \(\operatorname{rank}G(x_e)<n\). Then
	\(0\in\operatorname{int}V(x_e)\) never holds
	(Proposition~\ref{prop:empty-interior-underactuated}), whereas
	\[
	\mathcal{B}(x_e)\cap\operatorname{int}\Omega(x_e)\neq\emptyset
	\]
	may hold whenever \(f_0(x_e)\in\operatorname{Im}G(x_e)\) with the
	balancing coefficients strictly inside the input box. The
	\(\mathcal{B}\)-classification is therefore strictly more general than
	the ambient-space condition on \(V(x_e)\).
\end{proposition}

\begin{proof}
	The first claim is immediate from
	Proposition~\ref{prop:empty-interior-underactuated}: the ambient-space
	interior of \(V(x_e)\) in \(\mathbb{R}^n\) is empty whenever
	\(\operatorname{rank}G(x_e)<n\).
	
	For the second claim, take any \(x_e\) with
	\(f_0(x_e)\in\operatorname{Im}G(x_e)\). The balancing equation
	\[
	f_0(x_e)+\sum_i u_{e,i}g_i(x_e)=0
	\]
	admits at least one solution \(u_e\in\mathbb{R}^m\), so
	\(\mathcal{B}_e(x_e)\neq\emptyset\). When some solution lies in
	\(\operatorname{int}\Omega(x_e)\), the conditions of Case~(A) are met.
	
	As an explicit witness, take \(n=3\), \(m=2\),
	\[
	f_0(x_e)=(-2,-2,0),\quad
	g_1(x_e)=(1,0,0),\quad
	g_2(x_e)=(0,1,0),
	\]
	and
	\[
	\Omega(x_e)=[1,3]\times[1,3],
	\]
	which satisfies the zero-excluding positive subcase of
	Assumption~\ref{ass:1}. Then
	\[
	u_e=(2,2)\in
	\mathcal{B}(x_e)\cap\operatorname{int}\Omega(x_e),
	\]
	so \(x_e\) belongs to Case~(A). However,
	\begin{align*}
		V(x_e) &=
		\{(u_1-2,u_2-2,0): u_i\in[1,3]\}\\
		&=
		[-1,1]\times[-1,1]\times\{0\}.
	\end{align*}
	Thus \(V(x_e)\) is a two-dimensional square embedded in a
	two-dimensional affine subspace of \(\mathbb{R}^3\), with empty
	interior in \(\mathbb{R}^3\). The ambient-space condition
	\(0\in\operatorname{int}V(x_e)\) fails, despite the
	\(\mathcal{B}\)-classification placing \(x_e\) in the favorable
	regime.
	
	The same rank-deficiency phenomenon occurs in the planar CDPR of
	Section~\ref{sec:application}, where
	\(\operatorname{rank}G(x_e)=2<6=n\). At reference configurations for
	which the balancing tensions lie in the interior of the admissible
	box, the \(\mathcal{B}\)-classification identifies Case~(A) even
	though \(V(x_e)\) has empty interior in \(\mathbb{R}^6\).
\end{proof}

Propositions~\ref{prop:coincidence-fully-actuated}--%
\ref{prop:strict-generality-underactuated} together establish that the
\(\mathcal{B}\)-classification provides a single geometric framework
covering both structural regimes. In the full-row-rank regime
\(\operatorname{rank}G(x_e)=n\), Cases~(A), (B), and (C) coincide
pointwise with
\[
0\in\operatorname{int}V(x_e),\qquad
0\notin V(x_e),\qquad
0\in\partial V(x_e),
\]
respectively, so the position of the origin relative to \(V(x_e)\) is
recovered as a specialization. In the rank-deficient regime
\(\operatorname{rank}G(x_e)<n\), the ambient condition
\(0\in\operatorname{int}V(x_e)\) becomes vacuous, but the
\(\mathcal{B}\)-classification continues to distinguish the
interior-balancing, no-balancing, and boundary-balancing regimes.

The ambient-space condition on \(V(x_e)\) is therefore a particular
case of the input-space \(\mathcal{B}\)-classification, recovered when
\(\operatorname{rank}G(x_e)=n\). In the sequel, we use the
\(\mathcal{B}\)-classification to analyze the interior-balancing and
no-balancing regimes, while the boundary-balancing regime is identified
but left for future work.
\section{Case (A): Interior Admissible Balancing Inputs}
\label{sec:case-A}

This section addresses the regime
\[
\mathcal{B}(x_e)\cap\operatorname{int}\Omega(x_e)\neq\emptyset,
\]
which is the input-space counterpart of the favorable ambient
condition \(0\in\operatorname{int}V(x_e)\) in the zero-excluding
setting. The presence of an interior admissible balancing input
\(u_e\) allows a constant input shift, uniformly admissible over a
state-neighborhood,
\[
u=u_e+\eta,
\]
(Proposition~\ref{prop:local-signed-variations}). This shift locally
embeds a symmetric-control system into the original constrained
dynamics. Classical accessibility and STLC criteria can then be applied
to the shifted system: LARC yields local accessibility via
Krener~\cite{krener1974generalization}
(Corollary~\ref{cor:interior-accessibility}), and the weighted
\(\mathcal{S}(\theta)\)-condition yields STLC via
Sussmann~\cite{sussmann1987general}
(Corollary~\ref{cor:caseA-stlc}).

The input shift is classical in the constant-domain
setting~\cite{coron2007control,bullo2005geometric}. Its
non-triviality here lies in the \emph{uniform} admissibility of the
shift across a state-neighborhood, made possible by continuity of the
bound functions (Assumption~\ref{ass:1}), and in its precise
\emph{scope}: it succeeds in Case~(A), where an interior admissible
balancing input is available, and the shift-to-equilibrium reduction is
structurally impossible in Case~(B), where no admissible balancing input
exists (Section~\ref{sec:case-B}). The boundary-balancing regime
Case~(C) is different: an input shift around a boundary balancing input
leads to a one-sided tangent-cone control set rather than a symmetric
box, and is left for future work. The \(\mathcal{B}\)-classification
can thus be read operationally as a classification of the regimes in
which the input-shift reduction to a symmetric-control problem is
available.

\begin{proposition}[Uniform admissibility of the input shift]
	\label{prop:local-signed-variations}
	Suppose there exists
	\[
	u_e\in\mathcal{B}(x_e)\cap\operatorname{int}\Omega(x_e).
	\]
	Then there exist a neighborhood \(\mathcal{N}_{x_e}\) of \(x_e\) and a
	constant \(\varepsilon>0\) such that
	\begin{equation}
		\label{eq:uniform-shift-inclusion}
		u_e+[-\varepsilon,\varepsilon]^m\subset\Omega(x),
		\qquad x\in\mathcal{N}_{x_e}.
	\end{equation}
	Under the change of variables \(u=u_e+\eta\), the
	system~\eqref{eq:system_dynamique} contains, on
	\(\mathcal{N}_{x_e}\), the shifted symmetric-control system
	\begin{equation}
		\label{eq:shifted-system}
		\dot{x}
		=
		\tilde f_0(x)+\sum_{i=1}^m \eta_i g_i(x),
		\qquad
		\eta\in[-\varepsilon,\varepsilon]^m,
	\end{equation}
	where
	\[
	\tilde f_0(x):=f_0(x)+\sum_{i=1}^m u_{e,i}g_i(x)
	\]
	satisfies \(\tilde f_0(x_e)=0\).
\end{proposition}

\begin{proof}
	Since \(u_e\in\operatorname{int}\Omega(x_e)\), there exists
	\(\varepsilon_0>0\) such that
	\[
	\underline u_i(x_e)<u_{e,i}-\varepsilon_0,
	\qquad
	u_{e,i}+\varepsilon_0<\overline u_i(x_e)
	\]
	for every \(i=1,\ldots,m\). By continuity of
	\(\underline u_i\) and \(\overline u_i\)
	(Assumption~\ref{ass:1}), there exists a neighborhood
	\(\mathcal{N}_{x_e}\) of \(x_e\) on which
	\[
	\underline u_i(x)<u_{e,i}-\varepsilon_0/2,
	\qquad
	u_{e,i}+\varepsilon_0/2<\overline u_i(x)
	\]
	for every \(i=1,\ldots,m\). Setting
	\(\varepsilon:=\varepsilon_0/2\) yields
	\eqref{eq:uniform-shift-inclusion}. The equilibrium claim
	\(\tilde f_0(x_e)=0\) follows from
	\(u_e\in\mathcal{B}_e(x_e)\).
\end{proof}

\begin{lemma}[Lie-algebra invariance under shift]
	\label{lem:lie-algebra-shift}
	The Lie algebras
	\[
	\mathcal{L}(\{f_0,g_1,\ldots,g_m\})
	\qquad\text{and}\qquad
	\mathcal{L}(\{\tilde f_0,g_1,\ldots,g_m\})
	\]
	coincide pointwise on \(\mathbb{R}^n\). In particular, LARC at
	\(x_e\) holds for the original family if and only if it holds for the
	shifted family.
\end{lemma}

\begin{proof}
	This is a standard invariance under constant linear
	feedback~\cite[\S3.5]{coron2007control}. Since
	\[
	\tilde f_0=f_0+\sum_i u_{e,i}g_i
	\]
	is a constant linear combination of the original family, every
	iterated bracket of
	\(\{\tilde f_0,g_1,\ldots,g_m\}\) expands into a constant linear
	combination of iterated brackets of
	\(\{f_0,g_1,\ldots,g_m\}\), and conversely. The generated Lie
	subalgebras therefore agree pointwise.
\end{proof}

\begin{corollary}[Accessibility under an interior balancing input]
	\label{cor:interior-accessibility}
	Assume that
	\[
	\mathcal{B}(x_e)\cap\operatorname{int}\Omega(x_e)\neq\emptyset,
	\]
	and that the system \((\mathcal{F},\Omega)\) satisfies the Lie Algebra
	Rank Condition at \(x_e\). Then \((\mathcal{F},\Omega)\) is locally
	accessible from \(x_e\) under \(\Omega\).
\end{corollary}

\begin{proof}
	By Proposition~\ref{prop:local-signed-variations}, there exist a
	neighborhood \(\mathcal{N}_{x_e}\) and \(\varepsilon>0\) such that the
	shifted system~\eqref{eq:shifted-system} with symmetric control set
	\([-\varepsilon,\varepsilon]^m\) is admissible for the original
	constrained system on \(\mathcal{N}_{x_e}\). Indeed, every trajectory
	of the shifted system staying in \(\mathcal{N}_{x_e}\) is also a
	trajectory of \((\mathcal{F},\Omega)\) via the inverse change of
	variables \(u=u_e+\eta\). The shifted control set contains \(0\) in its
	interior, and \(\tilde f_0(x_e)=0\), so \(x_e\) is an equilibrium of
	the shifted system.
	
	By Lemma~\ref{lem:lie-algebra-shift}, LARC for the original family at
	\(x_e\) implies LARC for the shifted family at \(x_e\). Krener's
	accessibility theorem with drift~\cite{krener1974generalization},
	applied to the shifted system on \([-\varepsilon,\varepsilon]^m\),
	ensures that the reachable set of the shifted system from \(x_e\) in
	\(\mathcal{N}_{x_e}\) has non-empty interior in \(\mathbb{R}^n\) for
	every \(T>0\). Since the shifted reachable set is contained in the
	reachable set of the original constrained system, the latter also has
	non-empty interior for every \(T>0\). Hence
	\((\mathcal{F},\Omega)\) is locally accessible from \(x_e\).
\end{proof}

\begin{remark}[Propagation to nearby states]
	\label{rem:propagation-accessibility}
	If LARC holds at \(x_e\), then there exists a finite collection of Lie
	brackets whose values span \(\mathbb{R}^n\) at \(x_e\). By continuity
	of these bracket vector fields, the same collection remains of full
	rank on a neighborhood of \(x_e\). Hence LARC propagates locally.
	Local accessibility from a nearby state \(x_0\) requires, in addition,
	the existence of an interior admissible balancing input at \(x_0\).
	This is automatic on a neighborhood when the algebraic balancing
	equation admits a continuous admissible selection
	\(x\mapsto u_e(x)\), as in the planar CDPR application of
	Section~\ref{sec:application}, where \(\mathcal{B}_e(x)\) is the
	singleton determined by Cramer's rule.
\end{remark}

Local accessibility is strictly weaker than STLC: a system may be
locally accessible without being STLC when forward reachability is
biased in a particular direction. The following corollary upgrades
Corollary~\ref{cor:interior-accessibility} to STLC via Sussmann's
higher-order condition~\cite[Thm.~7.3]{sussmann1987general}.

\begin{corollary}[STLC under interior balancing and Sussmann's condition]
	\label{cor:caseA-stlc}
	Assume that the family
	\[
	\mathcal{F}=\{f_0,g_1,\ldots,g_m\}
	\]
	is real-analytic on \(\mathbb{R}^n\), that
	\[
	\mathcal{B}(x_e)\cap\operatorname{int}\Omega(x_e)\neq\emptyset,
	\]
	and that the shifted family
	\[
	\tilde{\mathcal{F}}:=\{\tilde f_0,g_1,\ldots,g_m\}
	\]
	at \(x_e\) satisfies Sussmann's
	\(\mathcal{S}(\theta)\)-condition~\cite[Thm.~7.3]{sussmann1987general}
	for some \(\theta\in(0,1]\). Then \((\mathcal{F},\Omega)\) is
	small-time locally controllable at \(x_e\).
\end{corollary}

\begin{proof}
	By Proposition~\ref{prop:local-signed-variations}, the change of
	variables \(u=u_e+\eta\) embeds, on a state-neighborhood
	\(\mathcal{N}_{x_e}\), the shifted control-affine system
	\eqref{eq:shifted-system} with control set
	\([-\varepsilon,\varepsilon]^m\) into the original constrained system
	\((\mathcal{F},\Omega)\). The shifted system has an analytic
	control-affine structure inherited from \(\mathcal{F}\), a controlled
	equilibrium at \(x_e\) since \(\tilde f_0(x_e)=0\), and a symmetric
	compact convex control set containing \(0\) in its interior. The
	hypotheses of Sussmann's theorem are therefore met, and the assumed
	\(\mathcal{S}(\theta)\)-condition implies STLC of the shifted system
	at \(x_e\): for every \(\delta>0\), there exists \(r(\delta)>0\) such
	that
	\[
	B_{r(\delta)}(x_e)
	\subseteq
	\mathcal{R}_{\tilde{\mathcal{F}}}^{[-\varepsilon,\varepsilon]^m}
	(\leq \delta,x_e;\mathcal{N}_{x_e}).
	\]
	Since every shifted trajectory staying in \(\mathcal{N}_{x_e}\) is an
	admissible trajectory of the original constrained system, the same
	inclusion holds for
	\[
	\mathcal{R}_{\mathcal{F}}^{\Omega}
	(\leq \delta,x_e;\mathcal{N}_{x_e}).
	\]
	This is precisely the STLC condition of Definition~\ref{def:stlc}.
\end{proof}

\begin{remark}[Role of Sussmann's \(\mathcal{S}(\theta)\)-condition]
	\label{rem:sussmann-condition}
	For \(\theta\in(0,1]\), the \(\mathcal{S}(\theta)\)-condition
	of~\cite[Thm.~7.3]{sussmann1987general} requires that every
	\(\theta\)-bad bracket \(\beta\) of \(\tilde{\mathcal{F}}\) at
	\(x_e\) be a linear combination of \(\theta\)-good brackets
	\(\gamma\) with
	\[
	\|\gamma\|_\theta<\|\beta\|_\theta,
	\]
	where
	\[
	\|\beta\|_\theta
	:=
	\theta\,\delta_0(\beta)+\sum_i\delta_i(\beta)
	\]
	denotes the \(\theta\)-degree. Here, bad brackets are those with odd
	degree in \(\tilde f_0\) and even degree in each \(g_i\). The condition
	is automatic when \(\{g_1,\ldots,g_m\}\) generates the full tangent
	space at \(x_e\), or when
	\(\operatorname{span}\{g_1,\ldots,g_m\}\) is involutive of full rank.
	For the planar underactuated CDPR of Section~\ref{sec:application},
	neither specialization holds; the condition fails at \(\theta=1\) but
	holds for every \(\theta\in(0,1/2)\) via the iterated drift fields
	\[
	\operatorname{ad}^k(\tilde f_0)g_i,
	\qquad k=1,2,3,
	\]
	as verified in the supplementary material.
\end{remark}

\begin{remark}[Dependence on the choice of \(u_e\)]
	\label{rem:choice-ubar}
	When
	\[
	\mathcal{B}(x_e)\cap\operatorname{int}\Omega(x_e)
	\]
	is not a singleton, for instance in over-actuated cases with
	\(\operatorname{rank}G(x_e)<m\), the shift target \(u_e\) is not
	unique. Lemma~\ref{lem:lie-algebra-shift} implies that LARC of the
	shifted family, and hence Corollary~\ref{cor:interior-accessibility},
	are independent of the choice of \(u_e\). By contrast, the
	\(\mathcal{S}(\theta)\)-condition for the shifted family
	\(\tilde{\mathcal{F}}\) depends on \(u_e\) through \(\tilde f_0\).
	Corollary~\ref{cor:caseA-stlc} asserts STLC as soon as the condition
	holds for \emph{some} admissible interior input
	\[
	u_e\in\mathcal{B}(x_e)\cap\operatorname{int}\Omega(x_e).
	\]
\end{remark}

The favorable Case~(A), where an interior admissible balancing input
enables a uniform shift to a symmetric problem, contrasts sharply with
the opposite regime in which no admissible balancing input exists at
all. The next section addresses this regime, which constitutes the
central contribution of the paper.
\section{Case (B): No Admissible Balancing Input}
\label{sec:case-B}

This section establishes the main technical obstruction of the paper.
In Case~(A), the existence of an admissible balancing input interior to
\(\Omega(x_e)\) enabled a uniform shift to a symmetric-control problem,
allowing classical accessibility and STLC criteria to apply. We now
address the opposite regime:
\[
\mathcal{B}(x_e)=\emptyset,
\qquad\text{equivalently}\qquad
0\notin V(x_e).
\]
No admissible input can hold the state at \(x_e\), even
instantaneously, and the feasible velocity set is strictly separated
from the origin.

The argument proceeds in three steps. First, the projection of the
origin onto the compact convex set \(V(x_e)\) produces a separating
covector \(\lambda\) and a strictly positive distance
\[
\alpha=\operatorname{dist}(0,V(x_e))
\]
(Lemma~\ref{lem:projection}). Second, by Hausdorff continuity of the
set-valued map \(V(\cdot)\), this separation extends uniformly to a
neighborhood of \(x_e\), with a possibly smaller rate
\(\alpha'\leq\alpha\) (Lemma~\ref{lem:uniformization}). Third, the
linear functional
\[
\Phi(x):=\langle\lambda,x\rangle
\]
becomes a local barrier: it increases along every admissible trajectory
remaining in the obstruction neighborhood, at rate at least
\(\alpha'\). This rules out local reachability of points satisfying
\(\Phi(y)<\Phi(x_e)\), and therefore rules out STLC
(Theorem~\ref{thm:no-balancing-nonstlc}).

The key point is that the separation is performed in the instantaneous
feasible velocity set \(V(x_e)\), rather than in a time-integrated
reachable set. This is the main departure from Brammer's classical
separating-hyperplane argument~\cite{brammer1972controllability}, and
is what makes the obstruction applicable to nonlinear control-affine
systems with state-dependent constraints at reference states that need
not be controlled equilibria of the constrained system.

\subsection{Geometric Setup: Projection and Separation}
\label{subsec:caseB-projection}

The first ingredient is the projection of the origin onto \(V(x_e)\).
Since \(V(x_e)\) is convex and compact
(Lemma~\ref{lem:V-properties}), this projection is unique. Since
\(0\notin V(x_e)\), the projection is nonzero and yields a strict
separation.

\begin{lemma}[Projection of the origin onto \(V(x_e)\)]
	\label{lem:projection}
	Let \(x_e\in\mathbb{R}^n\) satisfy \(0\notin V(x_e)\). The
	minimization problem
	\[
	\min_{v\in V(x_e)} \frac{1}{2}\|v\|^2
	\]
	admits a unique minimizer \(v^\star\neq 0\). Setting
	\[
	\alpha:=\|v^\star\|>0,
	\qquad
	\lambda:=\frac{v^\star}{\|v^\star\|}\in(\mathbb{R}^n)^*,
	\]
	where \(\mathbb{R}^n\) is identified with its dual through the
	Euclidean pairing, one has
	\begin{equation}
		\label{eq:separation}
		\langle \lambda,v\rangle \geq \alpha
		\qquad
		\forall v\in V(x_e).
	\end{equation}
	The constant \(\alpha=\operatorname{dist}(0,V(x_e))\) is the Euclidean
	distance from the origin to \(V(x_e)\).
\end{lemma}

\begin{proof}
	Since \(V(x_e)\) is nonempty, compact, and convex
	(Lemma~\ref{lem:V-properties}), and since
	\(v\mapsto \frac{1}{2}\|v\|^2\) is continuous and strictly convex, the
	minimization problem admits a unique minimizer
	\(v^\star\in V(x_e)\). Since \(0\notin V(x_e)\), one has
	\(v^\star\neq 0\). The variational inequality for the Euclidean
	projection gives
	\[
	\langle v^\star,v-v^\star\rangle\geq 0
	\qquad
	\forall v\in V(x_e).
	\]
	Hence
	\[
	\langle v^\star,v\rangle\geq \|v^\star\|^2
	\qquad
	\forall v\in V(x_e).
	\]
	Dividing by \(\|v^\star\|>0\) yields
	\eqref{eq:separation}.
\end{proof}

The covector \(\lambda\) encodes the first-order obstruction: every
admissible velocity at \(x_e\) has a positive component along
\(\lambda\), at least \(\alpha>0\). To turn this pointwise separation
into a statement along trajectories, the separation must persist for
nearby states.

\subsection{Local Uniformization}
\label{subsec:caseB-uniformization}

An admissible trajectory starting at \(x_e\) generally leaves \(x_e\),
and its velocity at time \(t\) belongs to \(V(x(t))\), not to
\(V(x_e)\). The next lemma shows that the separation obtained at
\(x_e\) remains valid on a neighborhood, possibly with a smaller
constant.

\begin{lemma}[Local uniformization]
	\label{lem:uniformization}
	Let \(x_e\in\mathbb{R}^n\) satisfy \(0\notin V(x_e)\), and let
	\(\lambda\) and \(\alpha>0\) be as in
	Lemma~\ref{lem:projection}. There exist \(\rho>0\) and
	\(\alpha'\in(0,\alpha]\) such that
	\begin{equation}
		\label{eq:uniformization}
		\inf_{v\in V(x)}\langle\lambda,v\rangle\geq \alpha'
		\qquad
		\forall x\in B_\rho(x_e).
	\end{equation}
\end{lemma}

\begin{proof}
	Define
	\[
	h(x):=\inf_{v\in V(x)}\langle\lambda,v\rangle .
	\]
	Equivalently,
	\[
	h(x)=-\sup_{v\in V(x)} f(x,v),
	\qquad
	f(x,v):=-\langle\lambda,v\rangle .
	\]
	The function \(f\) is continuous, and by
	Lemma~\ref{lem:V-properties}, \(V(\cdot)\) is Hausdorff continuous with
	nonempty compact values. Berge's maximum theorem implies that
	\(h\) is continuous. By
	Lemma~\ref{lem:projection}, \(h(x_e)\geq \alpha>0\). Hence, by
	continuity of \(h\) at \(x_e\), there exist \(\rho>0\) and
	\(\alpha'\in(0,\alpha]\) such that \(h(x)\geq\alpha'\) for all
	\(x\in B_\rho(x_e)\).
\end{proof}

The constant \(\alpha'\) is a local robustness margin for the
separation. It satisfies \(\alpha'\leq\alpha\), with equality in
situations where the feasible velocity set does not deteriorate near
\(x_e\), for instance in constant-velocity-set linear specializations.

\subsection{Main Result: Barrier Functional and Failure of STLC}
\label{subsec:caseB-main}

We now prove that the no-balancing condition produces a local
one-sided barrier. Let
\[
\Phi(x):=\langle\lambda,x\rangle .
\]
By Lemma~\ref{lem:uniformization}, \(\Phi\) increases strictly along
every admissible trajectory that remains in \(B_\rho(x_e)\). Hence
points lying on the side
\[
\Phi(y)<\Phi(x_e)
\]
are locally unreachable, which contradicts the ball-inclusion required
by STLC.

\begin{theorem}[Strict monotonicity and failure of STLC]
	\label{thm:no-balancing-nonstlc}
	Assume \(\mathcal{B}(x_e)=\emptyset\), equivalently
	\(0\notin V(x_e)\). Let \(\lambda\), \(\rho\), and
	\(\alpha'>0\) be given by Lemma~\ref{lem:uniformization}, and define
	\[
	\Phi(x):=\langle\lambda,x\rangle .
	\]
	Then every admissible trajectory \(x(\cdot)\) of
	\eqref{eq:system_dynamique} that remains in \(B_\rho(x_e)\) satisfies,
	for almost every \(t\),
	\begin{equation}
		\label{eq:phi-dot}
		\dot\Phi(x(t))
		=
		\langle\lambda,\dot x(t)\rangle
		\geq
		\alpha'>0.
	\end{equation}
	Consequently, for all \(t\) such that
	\(x([0,t])\subset B_\rho(x_e)\),
	\begin{equation}
		\label{eq:phi-bound}
		\Phi(x(t))
		\geq
		\Phi(x(0))+\alpha' t .
	\end{equation}
	In particular, no point \(y\) satisfying
	\(\Phi(y)<\Phi(x_e)\) can be reached from \(x_e\) by an admissible
	trajectory remaining in \(B_\rho(x_e)\). Therefore the system is not
	STLC at \(x_e\).
\end{theorem}

\begin{proof}
	Let \(x(\cdot)\) be an admissible trajectory of
	\eqref{eq:system_dynamique} such that \(x(t)\in B_\rho(x_e)\) for the
	times under consideration. For almost every such \(t\),
	\[
	\dot x(t)\in V(x(t)).
	\]
	Using Lemma~\ref{lem:uniformization}, we obtain
	\[
	\dot\Phi(x(t))
	=
	\langle\lambda,\dot x(t)\rangle
	\geq
	\inf_{v\in V(x(t))}\langle\lambda,v\rangle
	\geq
	\alpha',
	\]
	which proves \eqref{eq:phi-dot}. Integrating over \([0,t]\) gives
	\eqref{eq:phi-bound}.
	
	It remains to prove that STLC fails. Let \(\mathcal{N}_{x_e}\) be an
	arbitrary open neighborhood of \(x_e\). Set
	\[
	\mathcal{N}'_{x_e}:=\mathcal{N}_{x_e}\cap B_\rho(x_e),
	\]
	which is an open neighborhood of \(x_e\). Choose \(r_0>0\) such that
	\[
	\overline{B_{r_0}(x_e)}\subset \mathcal N'_{x_e}.
	\]
	Define
	\[
	L :=
	\sup_{x\in\overline{B_\rho(x_e)}}\sup_{v\in V(x)}\|v\|<\infty.
	\]
	This quantity is finite by compactness of the closed ball
	\(\overline{B_\rho(x_e)}\), compact-valuedness of \(V(\cdot)\), and
	Hausdorff continuity (Lemma~\ref{lem:V-properties}). Set
	\[
	\varepsilon_0:=\min\left\{\frac{r_0}{2L},1\right\}>0.
	\]
	
	We claim that every admissible trajectory \(x(\cdot)\) with
	\(x(0)=x_e\), remaining in \(\mathcal{N}_{x_e}\) on
	\([0,\varepsilon]\), with \(0<\varepsilon\leq\varepsilon_0\), actually
	remains in \(B_\rho(x_e)\) on \([0,\varepsilon]\). Suppose otherwise,
	and let \(\tau\) be the first time at which
	\(\|x(\tau)-x_e\|=r_0\). For all \(t\in[0,\tau]\), the trajectory lies
	in \(\overline{B_{r_0}(x_e)}\subset B_\rho(x_e)\), so
	\(\|\dot x(t)\|\leq L\) for almost every \(t\in[0,\tau]\). Hence
	\[
	r_0
	=
	\|x(\tau)-x_e\|
	\leq
	\int_0^\tau \|\dot x(s)\|\,ds
	\leq
	L\tau
	\leq
	L\varepsilon_0
	\leq
	\frac{r_0}{2},
	\]
	a contradiction. Therefore every such trajectory remains in
	\(B_\rho(x_e)\), and the barrier estimate \eqref{eq:phi-bound} applies.
	
	Now fix \(0<\varepsilon\leq\varepsilon_0\) and any candidate radius
	\(r(\varepsilon)>0\). Let
	\[
	\delta:=\frac{1}{2}\min\{r(\varepsilon),r_0\}>0
	\]
	and define
	\[
	y:=x_e-\delta\,\frac{\lambda}{\|\lambda\|}.
	\]
	Then
	\[
	y\in B_{r(\varepsilon)}(x_e)\cap B_{r_0}(x_e)
	\subset \mathcal{N}'_{x_e},
	\]
	and
	\[
	\Phi(y)
	=
	\Phi(x_e)-\delta\|\lambda\|
	<
	\Phi(x_e).
	\]
	By the preceding argument, any admissible trajectory starting from
	\(x_e\) and remaining in \(\mathcal{N}_{x_e}\) on
	\([0,\varepsilon]\) remains in \(B_\rho(x_e)\). Therefore
	\eqref{eq:phi-bound} gives, for every \(t\in(0,\varepsilon]\),
	\[
	\Phi(x(t))
	\geq
	\Phi(x_e)+\alpha't
	>
	\Phi(x_e)
	>
	\Phi(y).
	\]
	Thus \(y\) cannot be reached from \(x_e\) inside
	\(\mathcal{N}_{x_e}\) in time \(\leq\varepsilon\), although
	\(y\in B_{r(\varepsilon)}(x_e)\). This contradicts the STLC inclusion
	of Definition~\ref{def:stlc}. Since the candidate neighborhood
	\(\mathcal{N}_{x_e}\) was arbitrary, STLC fails at \(x_e\).
\end{proof}

Theorem~\ref{thm:no-balancing-nonstlc} is first-order, quantitative,
and robust. It is first-order because a single integration of the
velocity inequality suffices. It is quantitative because the escape
rate \(\alpha'\) is tied to the distance
\(\alpha=\operatorname{dist}(0,V(x_e))\). It is robust because the
separation persists throughout \(B_\rho(x_e)\) by local
uniformization.

\subsection{Discussion}
\label{subsec:caseB-discussion}

\begin{remark}[Comparison with Brammer's necessary condition]
	\label{rem:brammer-comparison}
	Theorem~\ref{thm:no-balancing-nonstlc} differs from Brammer's
	separating-hyperplane argument~\cite[Lemma~3.1]{brammer1972controllability}
	in three respects. First, Brammer separates the origin from the
	time-integrated reachable set
	\[
	\mathcal{R}_\infty=\bigcup_{t>0}\mathcal{R}(t),
	\]
	whereas the present obstruction separates the origin from the
	instantaneous feasible velocity set \(V(x_e)\). Second, Brammer's
	proof is intrinsically linear and uses spectral decompositions and
	asymptotic properties of the linear flow, whereas the present argument
	uses only convex geometry and set-valued continuity. Third, Brammer
	works under a balancing hypothesis of the form
	\(\Omega\cap\ker B\neq\emptyset\), while the present result addresses
	the complementary no-balancing regime \(\mathcal{B}(x_e)=\emptyset\).
\end{remark}

\begin{remark}[Time estimates and saturated controls]
	\label{rem:local-time-estimates}
	The lower bound \(\alpha'\) in Lemma~\ref{lem:uniformization} gives a
	lower bound on \(\dot\Phi\). An upper bound is also available. Define
	\[
	\beta:=
	\sup_{x\in\overline{B_\rho(x_e)}}\sup_{v\in V(x)}
	\langle\lambda,v\rangle .
	\]
	By compactness of \(\overline{B_\rho(x_e)}\), compact-valuedness and
	Hausdorff continuity of \(V(\cdot)\), and Berge's maximum theorem,
	one has \(\beta<\infty\). Hence every
	admissible trajectory remaining in \(B_\rho(x_e)\) satisfies
	\[
	\alpha'
	\leq
	\dot\Phi(x(t))
	\leq
	\beta
	\qquad\text{for a.e. }t.
	\]
	If \(y\in B_\rho(x_e)\) is reached from \(x_e\) in time \(T\), inside
	\(B_\rho(x_e)\), and if \(\Phi(y)>\Phi(x_e)\), then
	\[
	\frac{\Phi(y)-\Phi(x_e)}{\beta}
	\leq
	T
	\leq
	\frac{\Phi(y)-\Phi(x_e)}{\alpha'}.
	\]
	Moreover, for each fixed \(x\), the extrema of
	\[
	u\mapsto
	\left\langle
	\lambda,
	f_0(x)+\sum_i u_i g_i(x)
	\right\rangle
	\]
	over the box \(\Omega(x)\) are attained at vertices of \(\Omega(x)\).
	Thus, under the usual nondegeneracy assumptions of the Pontryagin
	maximum principle, time-optimal controls for the corresponding local
	problem are expected to saturate input bounds except possibly on
	switching surfaces~\cite{bonnard2003singular}. This is consistent with
	saturated tension profiles observed in rest-to-rest CDPR
	planning~\cite{ida2019}.
\end{remark}

\begin{remark}[Finite-time reachability beyond STLC]
	\label{rem:exit-and-ftlc}
	Theorem~\ref{thm:no-balancing-nonstlc} rules out local instantaneous
	reversibility inside \(B_\rho(x_e)\), but it does not rule out
	finite-time reachability through trajectories that leave this
	obstruction neighborhood. Any recovery of controllability after the
	loss of STLC must therefore be mediated by nonlocal admissible
	excursions, as formalized below.
\end{remark}

\subsection{Beyond STLC: Admissible Excursions and Finite-Time Controllability}
\label{subsec:excursions}

Theorem~\ref{thm:no-balancing-nonstlc} shows that, in the no-balancing
regime, the barrier functional
\[
\Phi(x)=\langle\lambda,x\rangle
\]
is strictly increasing along every admissible trajectory remaining in
\(B_\rho(x_e)\). This excludes infinitesimal reversibility, but not
finite-time reachability through admissible trajectories that exit the
local obstruction neighborhood, exploit a region where the
feasible-velocity geometry changes, and return toward a target close
to \(x_e\).

\begin{definition}[Finite-time local controllability with admissible excursions]
	\label{def:excursion-controllability}
	Let \(x_e\) be a reference state satisfying the obstruction of
	Theorem~\ref{thm:no-balancing-nonstlc}, and let \(B_\rho(x_e)\) be an
	associated obstruction neighborhood. The system
	\((\mathcal{F},\Omega)\) is said to be \emph{locally finite-time
		controllable with admissible excursions} at \(x_e\) if there exist an
	open neighborhood \(W_0\) of \(x_e\) and a finite time \(T>0\) such
	that, for every target \(y\in W_0\setminus\{x_e\}\), there exists an
	admissible trajectory \(x:[0,T_y]\to\mathbb{R}^n\), with
	\(T_y\leq T\), satisfying
	\[
	x(0)=x_e,
	\qquad
	x(T_y)=y.
	\]
	Moreover, for every target
	\[
	y\in W_0^-:=\{y\in W_0:\Phi(y)<\Phi(x_e)\},
	\]
	any such reaching trajectory necessarily leaves the obstruction
	neighborhood:
	\[
	x([0,T_y])\not\subset B_\rho(x_e).
	\]
\end{definition}

The time horizon \(T\) in
Definition~\ref{def:excursion-controllability} is finite but not
required to be arbitrarily small. This is the essential distinction
from STLC. The excursion requirement is not imposed artificially:
Theorem~\ref{thm:no-balancing-nonstlc} implies that any trajectory
confined to \(B_\rho(x_e)\) satisfies
\[
\Phi(x(t))\geq \Phi(x_e)+\alpha't>\Phi(x_e),
\qquad t>0.
\]
Hence a target satisfying \(\Phi(y)<\Phi(x_e)\), if reachable at all,
must be reached by a trajectory that exits the obstruction
neighborhood.

\begin{remark}[Relation to Hermann--Krener weak accessibility]
	\label{rem:HK-comparison}
	Definition~\ref{def:excursion-controllability} is a strong finite-time
	local controllability requirement. A weaker formulation, closer to the
	weak accessibility viewpoint of Hermann and
	Krener~\cite{hermann1977nonlinear}, would require only that the
	finite-time reachable set have nonempty interior near \(x_e\), for
	example
	\[
	\operatorname{int}\!\left(
	\mathcal{R}_{\mathcal{F}}^\Omega(\leq T,x_e)\cap W_0
	\right)\neq\emptyset .
	\]
	The present definition is stronger: it requires coverage of a whole
	neighborhood \(W_0\), except for the reference point itself.
\end{remark}

\begin{remark}[Lower bound on excursion time]
	\label{rem:excursion-time}
	Before leaving \(B_\rho(x_e)\), every admissible trajectory satisfies
	\[
	\Phi(x(t))\geq \Phi(x_e)+\alpha't .
	\]
	Independently, define
	\[
	L_\rho:=
	\sup_{x\in\overline{B_\rho(x_e)}}\sup_{v\in V(x)}\|v\|.
	\]
	As above, \(L_\rho<\infty\). If \(t_{\rm exit}\) is the first exit time
	from \(B_\rho(x_e)\), then
	\[
	\rho
	=
	\|x(t_{\rm exit})-x_e\|
	\leq
	\int_0^{t_{\rm exit}}\|\dot x(t)\|\,dt
	\leq
	L_\rho t_{\rm exit}.
	\]
	Consequently,
	\begin{equation}
		\label{eq:exit-time-bound}
		t_{\rm exit}\geq \frac{\rho}{L_\rho}.
	\end{equation}
	Thus admissible excursions cannot occur in arbitrarily small time,
	which is consistent with the failure of STLC at \(x_e\).
\end{remark}

\begin{remark}[Structure of admissible excursions]
	\label{rem:excursion-structure}
	Let \(\lambda_e\) denote the separating covector associated with the
	Case~(B) reference state \(x_e\). A natural mechanism for admissible
	excursions is to reach a region where the original covector no longer
	defines a monotone barrier, namely a region where
	\[
	\inf_{v\in V(x)}\langle\lambda_e,v\rangle<0.
	\]
	In such a region, the system admits instantaneous velocities that
	decrease the original barrier functional. This may occur when the
	trajectory reaches a Case~(A) or Case~(C) region, where
	\(0\in V(\cdot)\), but this is not necessary: the original separating
	covector may also cease to be monotone inside another Case~(B) region
	with a different feasible-velocity geometry.
	
	The prototypical excursion has the form
	\[
	x_e\longrightarrow z\longrightarrow y,
	\]
	where \(z\notin B_\rho(x_e)\) lies in a region where the original
	barrier can be relaxed or reversed, and \(y\) is a target close to
	\(x_e\).
\end{remark}

\begin{remark}[Computational perspective for the CDPR model]
	\label{rem:excursion-cdpr}
	The planar two-cable CDPR of Section~\ref{sec:application} provides a
	concrete setting for studying admissible excursions. At a Case~(B)
	reference configuration, the covector \(\lambda\), rate \(\alpha'\),
	radius \(\rho\), and velocity bound \(L_\rho\) can be computed from the
	model data. Case~(A) regions are natural candidates for the return leg
	of an excursion, since the shifted interior-balancing system of
	Corollary~\ref{cor:interior-accessibility} becomes available there.
	The numerical construction and verification of excursion trajectories
	\[
	x_e\longrightarrow z\longrightarrow y
	\]
	for specific Case~(B) configurations is left for future work.
\end{remark}

\section{Application to a Planar Underactuated CDPR}
\label{sec:application}

This section illustrates the proposed framework on a planar suspended
cable-driven parallel robot (CDPR). The example is representative of
the class targeted in this paper: the actuation is strictly positive,
since cables can only pull, and the system is underactuated, since two
cable tensions control a six-dimensional state.

\subsection{Model and Positive-Tension Constraints}
\label{subsec:cdpr-model}

Consider a homogeneous rigid equilateral triangular platform of side
length \(b\) and mass \(m_p\), evolving in the vertical plane under
gravity. The upper vertices \(A_1,A_2\) are connected to fixed ceiling
anchors
\[
D_1=(-d,0),\qquad D_2=(d,0),
\]
with \(d>b/2\), while the lower vertex is unactuated. The generalized
coordinates are
\[
q=(x,y,\phi)^\top\in\mathbb{R}^3,
\qquad
\xi=(q,\dot q)^\top\in\mathbb{R}^6.
\]
In the body-fixed frame centered at the platform center of mass \(G\),
the upper attachment points are
\[
a_i^{\rm loc}=(\mp b/2,h)^\top,
\qquad
h:=\frac{b}{2\sqrt{3}},
\qquad i=1,2,
\]
and the moment of inertia about \(G\) is $I_G=\frac{m_p b^2}{12}.$

\begin{figure}[t]
	\centering
	\includegraphics[width=\linewidth]{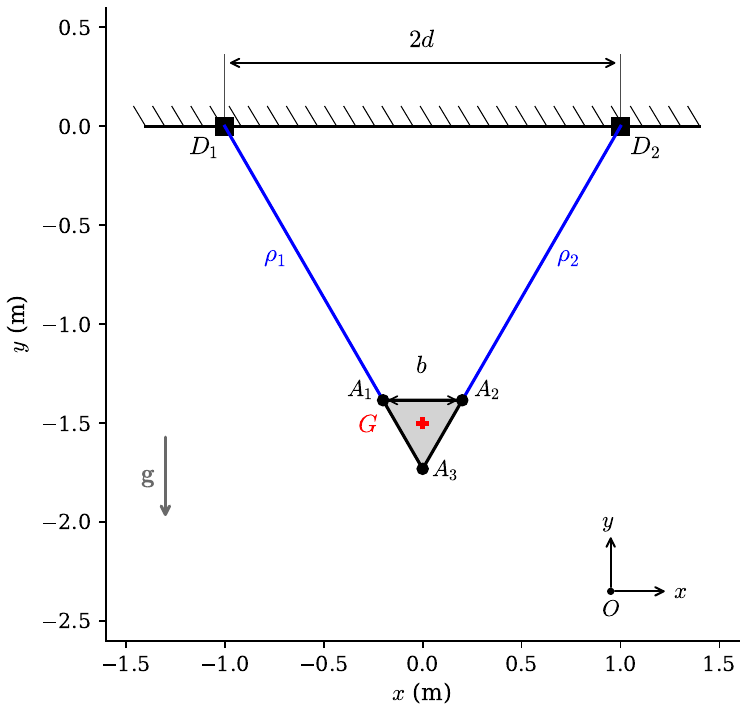}
	\caption{Planar two-cable CDPR with generalized coordinates
		\(q=(x,y,\phi)^\top\).}
	\label{fig:cdpr-schema}
\end{figure}

Cables transmit tensile forces only. Hence each physical tension
satisfies
\[
T_i\in[T_{\min},T_{\max}],
\qquad
0<T_{\min}<T_{\max},
\qquad i=1,2.
\]
The zero input is therefore excluded by construction. This places the
CDPR in the zero-excluding input regime considered throughout the
paper.

Let
\[
R(\phi)=
\begin{pmatrix}
	\cos\phi & -\sin\phi\\
	\sin\phi & \cos\phi
\end{pmatrix}.
\]
The position of the attachment point \(A_i\) is
\[
A_i(q)=(x,y)^\top+R(\phi)a_i^{\rm loc},
\]
and the cable vector from \(A_i\) to \(D_i\) is
\begin{equation}
	\label{eq:rho-def}
	\rho_i(q):=D_i-A_i(q),
	\qquad
	\ell_i(q):=\|\rho_i(q)\|.
\end{equation}
The model is considered on an operational domain where
\(\ell_i(q)>0\) for \(i=1,2\).

\subsection{Control-Affine Form}
\label{subsec:cdpr-affine-form}

The inertia matrix is $M=\operatorname{diag}(m_p,m_p,I_G),$
and the gravity vector is $s=(0,-m_p g,0)^\top.$
The cable force applied at \(A_i\) is $F_i=T_i\,\frac{\rho_i(q)}{\ell_i(q)}.$
Let $r_i(q):=R(\phi)a_i^{\rm loc}$
be the moment arm from the center of mass to \(A_i\). The scalar moment
generated by cable \(i\) is
\begin{equation}
	\label{eq:moment}
	\mathcal{M}_i(q)
	:=
	r_{i,x}(q)\rho_{i,y}(q)-r_{i,y}(q)\rho_{i,x}(q).
\end{equation}
The Euler--Lagrange equations can be written as
\[
M\ddot q=s+B_T(q)T,
\]
where \(T=(T_1,T_2)^\top\) and
\[
B_T(q)
=
\begin{pmatrix}
	\rho_{1,x}/\ell_1 & \rho_{2,x}/\ell_2\\
	\rho_{1,y}/\ell_1 & \rho_{2,y}/\ell_2\\
	\mathcal{M}_1/\ell_1 & \mathcal{M}_2/\ell_2
\end{pmatrix}.
\]

For the symbolic Lie-bracket computations, it is convenient to use the
normalized tensions
\begin{equation}
	\label{eq:normalized-tensions}
	\tau_i:=\frac{T_i}{\ell_i(q)},
	\qquad i=1,2.
\end{equation}
With \(\tau=(\tau_1,\tau_2)^\top\), the dynamics become
\[
M\ddot q=s+B(q)\tau,
\]
where
\begin{equation}
	\label{eq:B-matrix}
	B(q)
	=
	\begin{pmatrix}
		\rho_{1,x} & \rho_{2,x}\\
		\rho_{1,y} & \rho_{2,y}\\
		\mathcal{M}_1 & \mathcal{M}_2
	\end{pmatrix}.
\end{equation}
Thus the state dynamics admit the control-affine form
\begin{equation}
	\label{eq:cdpr-control-affine}
	\dot\xi=f_0(\xi)+\tau_1g_1(\xi)+\tau_2g_2(\xi),
\end{equation}
with
\begin{equation}
	\label{eq:f0-gi}
	f_0(\xi)=
	\begin{pmatrix}
		\dot q\\
		M^{-1}s
	\end{pmatrix}
	=
	\begin{pmatrix}
		\dot q\\
		0\\
		-g\\
		0
	\end{pmatrix},
	\qquad
	g_i(\xi)=
	\begin{pmatrix}
		0_3\\
		M^{-1}B_{:,i}(q)
	\end{pmatrix}.
\end{equation}
The vector fields are real-analytic on the operational domain.

The physical bounds \(T_i\in[T_{\min},T_{\max}]\) become the
state-dependent box
\begin{equation}
	\label{eq:cdpr-Omega}
	\Omega(q)
	=
	\prod_{i=1}^{2}
	\left[
	\frac{T_{\min}}{\ell_i(q)},
	\frac{T_{\max}}{\ell_i(q)}
	\right].
\end{equation}
Since \(\ell_i(q)>0\) and depends continuously on \(q\), the bound
functions in \eqref{eq:cdpr-Omega} are continuous and strictly
positive on the operational domain. Hence Assumption~\ref{ass:1}
holds. Moreover,
\[
n=6,\qquad m=2,\qquad \operatorname{rank}G(\xi)\leq 2<6,
\]
so the feasible velocity set \(V(\xi)\) has empty interior in
\(\mathbb{R}^6\). The ambient condition
\(0\in\operatorname{int}V(\xi)\) is therefore vacuous, while the
input-space \(\mathcal{B}\)-classification remains meaningful.

\subsection{Static Balancing Tensions}
\label{subsec:cdpr-equilibria}

A static reference state has the form $\xi_e=(q_e,0)^\top.$
At such a point, the balance equation associated with
\eqref{eq:cdpr-control-affine} reduces to
\begin{equation}
	\label{eq:static-balance}
	B(q_e)\tau_e=\Lambda,
	\qquad
	\Lambda:=(0,m_p g,0)^\top.
\end{equation}
This is a system of three scalar equations in two unknowns
\(\tau_e=(\tau_{e,1},\tau_{e,2})^\top\). Therefore a static balance
exists only if \(\Lambda\) belongs to the image of \(B(q_e)\).

Let
\[
\nu(q):=B_{:,1}(q)\times B_{:,2}(q).
\]
Then the compatibility condition is
\[
\Lambda^\top\nu(q_e)=0.
\]
Since \(\Lambda\) has only its second component nonzero, this reduces
to
\begin{equation}
	\label{eq:compatibility}
	\nu_y(q_e)
	=
	\mathcal{M}_1(q_e)\rho_{2,x}(q_e)
	-
	\rho_{1,x}(q_e)\mathcal{M}_2(q_e)
	=
	0.
\end{equation}
Accordingly, the static equilibrium manifold is
\begin{equation}
	\label{eq:equilibrium-manifold}
	\Sigma
	:=
	\{q\in\mathbb{R}^3:\nu_y(q)=0\}.
\end{equation}
On regular points of \(\Sigma\), this set is a two-dimensional
submanifold of the configuration space.

Assume now that the two cable directions are not collinear, so that
\[
\det B_{12}(q_e)
:=
\rho_{1,x}(q_e)\rho_{2,y}(q_e)
-
\rho_{1,y}(q_e)\rho_{2,x}(q_e)
\neq 0,
\]
where \(B_{12}\) is the \(2\times 2\) submatrix formed by the
translational rows of \(B(q_e)\). Then the first two equations of
\eqref{eq:static-balance} determine the unique balancing normalized
tensions:
\begin{equation}
	\label{eq:cramer-formula}
	\tau_{e,1}(q_e)
	=
	-\frac{m_p g\,\rho_{2,x}(q_e)}
	{\det B_{12}(q_e)},
	\quad
	\tau_{e,2}(q_e)
	=
	\frac{m_p g\,\rho_{1,x}(q_e)}
	{\det B_{12}(q_e)}.
\end{equation}
The third equation is satisfied precisely on \(\Sigma\). Therefore, at
a regular static reference state, the algebraic balancing set is the
singleton
\[
\mathcal{B}_e(\xi_e)=\{\tau_e(q_e)\}.
\]
The admissible balancing set is then
\[
\mathcal{B}(\xi_e)
=
\{\tau_e(q_e)\}\cap\Omega(q_e).
\]

Equivalently, in terms of physical tensions
\[
T_{e,i}(q_e):=\ell_i(q_e)\tau_{e,i}(q_e),
\]
the \(\mathcal{B}\)-classification is:
\[
\begin{array}{lll}
	\text{Case (A)} &\Longleftrightarrow&
	T_{e,i}(q_e)\in(T_{\min},T_{\max}),\quad i=1,2,\\[1mm]
	\text{Case (B)} &\Longleftrightarrow&
	\text{no admissible balancing tension exists},\\[1mm]
	\text{Case (C)} &\Longleftrightarrow&
	T_{e,i}(q_e)\in[T_{\min},T_{\max}],\ i=1,2,
	\text{ and at}\\
	& &\text{least one bound is active}.
\end{array}
\]
Thus the abstract classification of Section~\ref{sec:classification}
reduces, for the CDPR, to checking the position of the balancing cable
tensions relative to the admissible tension interval.

\subsection{Numerical Illustration of the \(\mathcal{B}\)-Classification}
\label{subsec:cdpr-numerical-classification}

We now illustrate the classification on the equilibrium manifold
\(\Sigma\). This numerical study is not intended as a trajectory-level
validation of finite-time controllability. Its purpose is to show that
the proposed \(\mathcal{B}\)-classification is computable and separates
the operating domain into the three regimes predicted by the theory.

We sample the operational set
\[
\mathcal{D}
:=
\{q_e\in\Sigma:\ |x_e|<d,\ |\phi_e|<\pi/2,\ y_e<-h\}.
\]
Away from the singular set of the parametrization, the compatibility
condition~\eqref{eq:compatibility} can be solved for \(y_e\) as a
function of \((x_e,\phi_e)\). Define
\[
\begin{aligned}
	N(x_e,\phi_e)
	:={}&
	\frac{b^2d}{4}\sin(2\phi_e)
	-bd^2\sin\phi_e
	-bhx_e
	+bx_e^2\sin\phi_e
	\\
	&\quad
	+2dh^2\sin\phi_e\cos\phi_e,
	\\[1mm]
	D(x_e,\phi_e)
	:={}&
	-bx_e\cos\phi_e
	+2dh\sin\phi_e .
\end{aligned}
\]
Then
\begin{equation}
	\label{eq:y-on-sigma}
	y_e(x_e,\phi_e)
	=
	-\frac{N(x_e,\phi_e)}{D(x_e,\phi_e)}.
\end{equation}
For each grid point \((x_e,\phi_e)\), we compute \(y_e\) from
\eqref{eq:y-on-sigma}, evaluate the balancing tensions from
\eqref{eq:cramer-formula}, recover the physical tensions
\[
T_{e,i}(q_e)=\ell_i(q_e)\tau_{e,i}(q_e),
\]
and classify the point according to Cases~(A), (B), and (C).

For the numerical illustration, we use
\[
b=0.5\,\mathrm{m},\quad
m_p=5\,\mathrm{kg},\quad
d=1.0\,\mathrm{m},\quad
g=9.81\,\mathrm{m/s^2}.
\]

\begin{table}[t]
	\centering
	\caption{Distribution of the \(\mathcal{B}\)-classification on
		\(\mathcal{D}\) for three actuator-bound scenarios. Sample size:
		\(|\mathcal{D}_h|=6536\) points.}
	\label{tab:classification-stats}
	\begin{tabular}{lcccc}
		\toprule
		\([T_{\min},T_{\max}]\) (N)
		& Region & Case (A) & Case (B) & Case (C) \\
		\midrule
		\multirow{3}{*}{\([15,50]\)}
		& On-axis  & \(81.5\%\) & \(18.0\%\) & \(0.5\%\) \\
		& Off-axis & \(27.3\%\) & \(68.8\%\) & \(3.9\%\) \\
		& Total    & \(29.0\%\) & \(67.3\%\) & \(3.8\%\) \\
		\midrule
		\multirow{3}{*}{\([27,50]\)}
		& On-axis  & \(47.5\%\) & \(36.5\%\) & \(16.0\%\) \\
		& Off-axis & \(6.3\%\)  & \(91.5\%\) & \(2.2\%\) \\
		& Total    & \(7.5\%\)  & \(89.8\%\) & \(2.6\%\) \\
		\midrule
		\multirow{3}{*}{\([30,40]\)}
		& On-axis  & \(19.0\%\) & \(75.0\%\) & \(6.0\%\) \\
		& Off-axis & \(0.4\%\)  & \(99.1\%\) & \(0.4\%\) \\
		& Total    & \(1.0\%\)  & \(98.4\%\) & \(0.6\%\) \\
		\bottomrule
	\end{tabular}
\end{table}

Table~\ref{tab:classification-stats} shows how the actuator bounds
shape the feasible static workspace. As the admissible tension interval
narrows, the Case~(A) region decreases from \(29.0\%\) of the sampled
domain for \([15,50]\) N to \(1.0\%\) for \([30,40]\) N. The
no-balancing regime becomes dominant under restrictive bounds,
reaching \(98.4\%\) of the sampled domain in the last scenario.

The symmetry axis \(\{x_e=0,\phi_e=0\}\) is significantly more
favorable than the off-axis region. This reflects the symmetric sharing
of the platform weight between the two cables. Off-axis configurations
induce asymmetric cable tensions, which more easily violate either the
lower or the upper actuator bound.

Finally, the position of \(T_{\min}\) relative to the asymptotic
symmetric tension \(m_p g/2\) explains the qualitative changes on the
axis. When \(T_{\min}<m_p g/2\), infeasibility on the axis is mainly
caused by excessive tension near the anchors. When
\(T_{\min}>m_p g/2\), low-tension infeasibility also appears, producing
additional boundary transitions between Cases~(A), (B), and (C).

\subsection{Implications for Accessibility and STLC}
\label{subsec:cdpr-implications}

The classification has direct controllability implications.

At a Case~(A) reference state \(\xi_e=(q_e,0)\), the balancing tension
\(\tau_e(q_e)\) lies in \(\operatorname{int}\Omega(q_e)\). Hence the
input shift
\[
\tau=\tau_e(q_e)+\eta
\]
is uniformly admissible on a neighborhood of \(\xi_e\). The shifted
system has an equilibrium at \(\xi_e\), and the results of
Section~\ref{sec:case-A} apply. In particular, if LARC holds at
\(\xi_e\), then the CDPR is locally accessible from \(\xi_e\) under the
positive-tension constraint. If, in addition, the shifted family
satisfies Sussmann's \(\mathcal{S}(\theta)\)-condition, then the CDPR
is STLC at \(\xi_e\). The Lie-bracket computations used for this
verification are reported in
the supplementary material.

At a Case~(B) reference state, no admissible cable tension balances the
dynamics. Equivalently,
\[
0\notin V(\xi_e).
\]
Theorem~\ref{thm:no-balancing-nonstlc} then gives a separating
covector \(\lambda\), a neighborhood \(B_\rho(\xi_e)\), and a rate
\(\alpha'>0\) such that the barrier functional
\[
\Phi(\xi)=\langle\lambda,\xi\rangle
\]
satisfies
\[
\dot\Phi(\xi(t))\geq \alpha'>0
\]
along every admissible trajectory remaining in \(B_\rho(\xi_e)\).
Therefore the CDPR is not STLC at such Case~(B) configurations.

This illustrates the role of the \(\mathcal{B}\)-classification in a
rank-deficient mechanical system: although
\(\operatorname{int}V(\xi_e)=\emptyset\) for every reference state, the
input-space classification still distinguishes configurations where
the classical shifted-control tools apply from configurations where a
local one-sided barrier obstructs STLC.

\subsection{Certified Numerical Validation of the No-Balancing Barrier}
\label{subsec:cdpr-barrier-verification}

We now provide a certified numerical validation of the no-balancing
obstruction in Theorem~\ref{thm:no-balancing-nonstlc}. The purpose is
twofold: first, to compute a rigorous lower bound for the barrier rate
on a continuous neighborhood of a representative Case~(B) state; and
second, to illustrate this certificate along extremal, adversarial, and
randomly switched admissible trajectories.

\subsubsection*{Reference state and separating geometry}

Consider the reference state
\[
\xi_e=(0,-1.5,0,0,0,0)^\top
\]
on the symmetry axis of \(\Sigma\), with physical tension bounds
\[
T_i\in[30,40]~\mathrm{N},\qquad i=1,2.
\]
The cable lengths at \(q_e\) satisfy
\[
\ell_1(q_e)=\ell_2(q_e)\simeq1.549~\mathrm{m}.
\]
Cramer's formula~\eqref{eq:cramer-formula} gives the balancing
physical tensions
\[
T_{e,1}=T_{e,2}\simeq28.0~\mathrm{N}<T_{\min}.
\]
Hence \(\mathcal{B}(\xi_e)=\emptyset\), and \(\xi_e\) belongs to
Case~(B).

The Euclidean projection of the origin onto \(V(\xi_e)\) is computed
by solving
\[
v^\star
=
\operatorname*{arg\,min}_{v\in V(\xi_e)}\|v\|.
\]
Since \(V(\xi_e)\) is the affine image of a two-dimensional box, the
projection problem can be solved exactly by examining its interior,
edges, and vertices. The optimum is attained at the minimum-tension
input
\[
T^\star=(30,30)^\top~\mathrm{N},
\]
and yields
\[
v^\star
=
(0,0,0,0,0.6902147,0)^\top,
\]
\[
\alpha=\|v^\star\|
=0.6902147~\mathrm{m\,s^{-2}}.
\]
Therefore,
\[
\lambda=\frac{v^\star}{\|v^\star\|}=e_5,
\qquad
\Phi(\xi)=\langle\lambda,\xi\rangle=\dot y.
\]

\subsubsection*{Interval certificate on a state neighborhood}

Because the state components have different physical units, we define
the dimensionless scaled norm
\[
\|\xi-\xi_e\|_S
:=
\|S^{-1}(\xi-\xi_e)\|_2,
\]
with
\[
S
=
\operatorname{diag}
\left(
1~\mathrm{m},
1~\mathrm{m},
1~\mathrm{rad},
1~\mathrm{m\,s^{-1}},
1~\mathrm{m\,s^{-1}},
1~\mathrm{rad\,s^{-1}}
\right).
\]
We consider the neighborhood
\[
\mathcal{N}_{0.05}
:=
\{\xi:\|\xi-\xi_e\|_S\leq0.05\}.
\]
Its configuration-space projection is contained in the box
\[
|x|\leq0.05~\mathrm{m},\qquad
|y+1.5|\leq0.05~\mathrm{m},\qquad
|\phi|\leq0.05~\mathrm{rad}.
\]

For \(\lambda=e_5\), the quantity
\[
\langle\lambda,f_0(\xi)+G_T(\xi)T\rangle
\]
is the vertical acceleration. It depends only on the configuration
variables and is affine in \(T\). Therefore its minimum over the
tension box is attained at one of the four vertices of
\([30,40]^2\).

We evaluate the resulting continuous minimization problem by outward-
rounded interval arithmetic combined with adaptive branch-and-bound
over the configuration box. With \(10^5\) terminal boxes, the
computation gives
\begin{equation}
	\label{eq:certified-alpha-prime}
	0.5865936
	\leq
	\min_{\substack{\xi\in\mathcal{N}_{0.05}\\
			T\in[30,40]^2}}
	\langle\lambda,f_0(\xi)+G_T(\xi)T\rangle
	\leq
	0.5906254
	\;\mathrm{m\,s^{-2}}
\end{equation}
Consequently, the certified rate
\[
\alpha'_{\mathrm{cert}}
:=
0.5865936~\mathrm{m\,s^{-2}}
\]
is valid for every admissible input and every state in
\(\mathcal{N}_{0.05}\). Theorem~\ref{thm:no-balancing-nonstlc}
therefore yields
\begin{equation}
	\label{eq:certified-barrier-CDPR}
	\Phi(\xi(t))-\Phi(\xi_e)
	\geq
	\alpha'_{\mathrm{cert}}t
\end{equation}
along every admissible trajectory that starts from \(\xi_e\) and
remains in \(\mathcal{N}_{0.05}\).

The interval calculation also gives the scaled velocity bound
\[
L_{0.05}\leq39.0894~\mathrm{s^{-1}},
\]
and hence the excursion-time estimate
\[
t_{\mathrm{exit}}
\geq
\frac{0.05}{L_{0.05}}
=1.279\times10^{-3}~\mathrm{s}.
\]
The bound is conservative because it is computed on the coordinate
box containing \(\mathcal{N}_{0.05}\), rather than on the scaled ball
itself.

\subsubsection*{Trajectory-level illustration}

We complement the interval certificate with \(1005\) admissible
profiles: the tension-box vertices, a state-dependent adversarial
input, uniformly sampled piecewise-constant controls, and randomly
switched bang--bang controls. All trajectories satisfy
\eqref{eq:certified-barrier-CDPR} while they remain in
\(\mathcal N_{0.05}\); the smallest observed margin is
\(2.07\times10^{-4}>0\). These simulations illustrate, but do not
establish, the universal certificate. The minimum-tension/adversarial
trajectory forms the lower envelope, while mixed and maximum tensions
leave the certified neighborhood sooner. Integration tolerances,
switching parameters, and trajectory data are provided in the
supplementary material.

\begin{figure*}[t]
	\centering
	\includegraphics[width=0.485\textwidth]
	{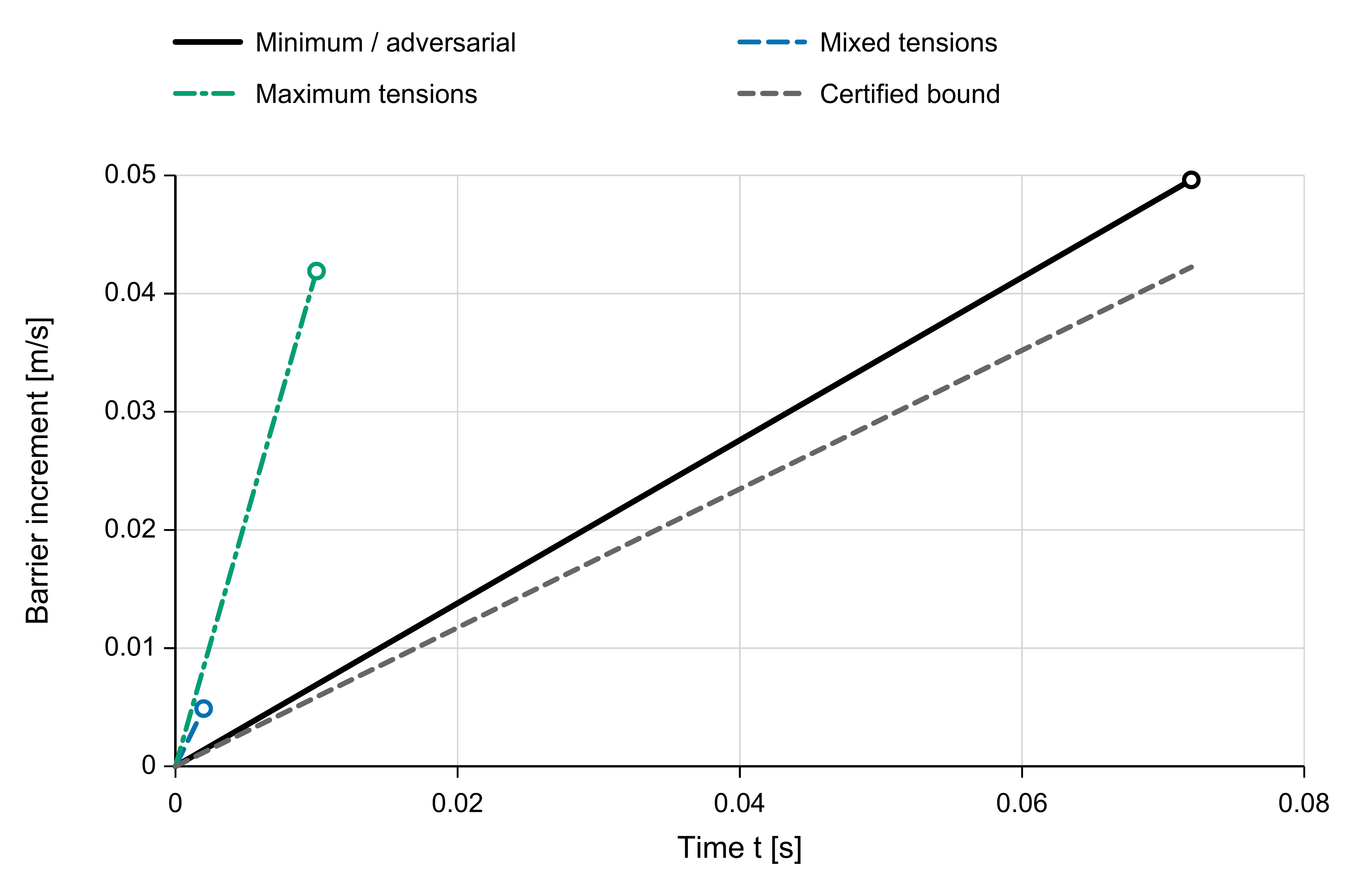}
	\hfill
	\includegraphics[width=0.485\textwidth]
	{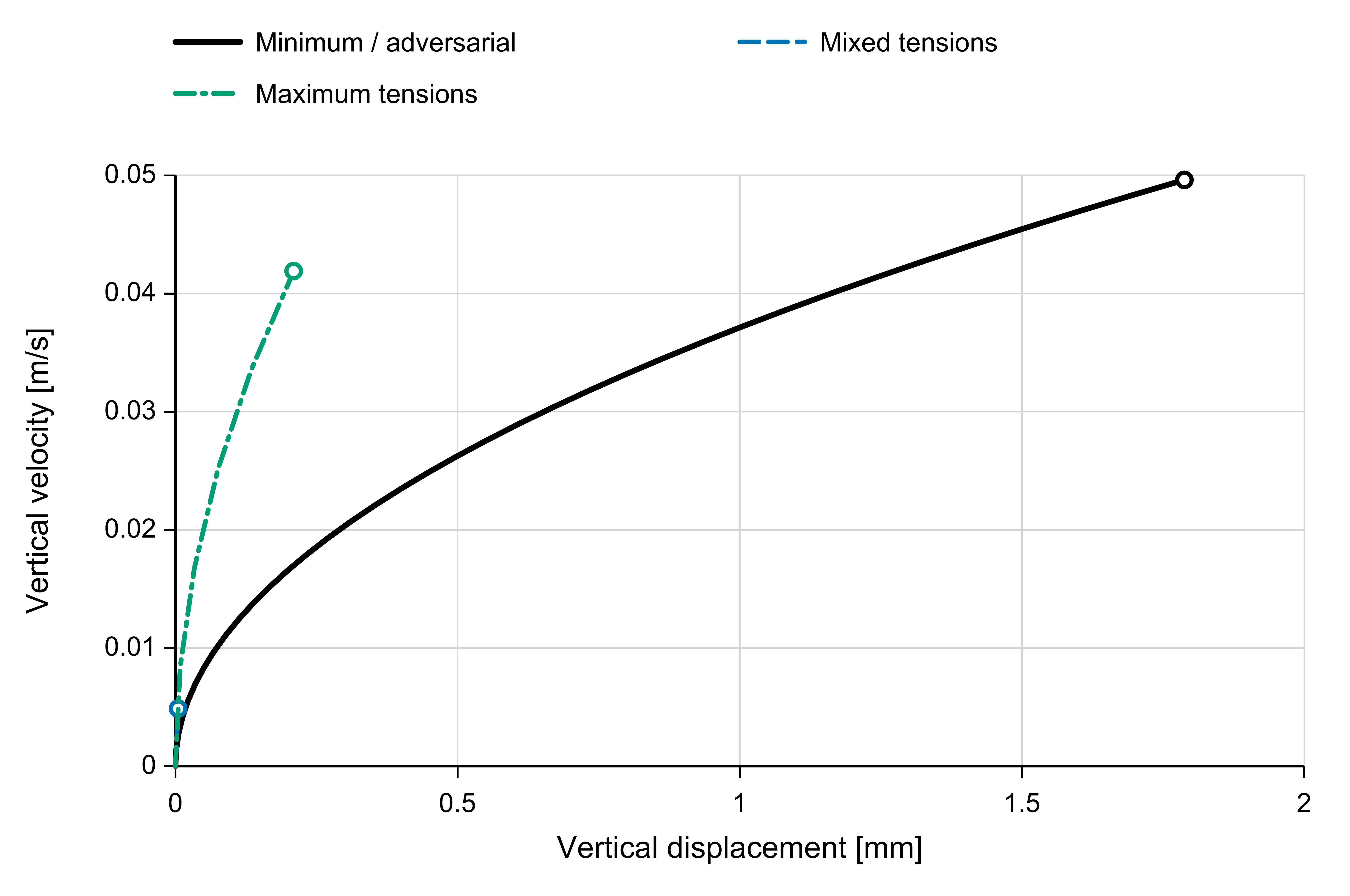}
	\caption{Certified barrier under \(T_i\in[30,40]\)~N.
		\emph{Left}: barrier increments and certified lower bound.
		\emph{Right}: vertical displacement--velocity projections. The
		interval computation certifies all admissible inputs on
		\(\mathcal N_{0.05}\).}
	\label{fig:certified-barrier-validation}
\end{figure*}

\subsection{Boundary Configurations and Admissible Excursions}
\label{subsec:cdpr-boundary-excursions}

For a regular static configuration, Case~(C) occurs when both
balancing tensions belong to \([T_{\min},T_{\max}]\) and at least one
bound is active. It is therefore the transition between
interior-balancing and no-balancing regions and requires a one-sided
tangent-cone analysis. The certified barrier is local and does not
exclude finite-time reachability through excursions leaving
\(\mathcal N_{0.05}\); constructing such excursions remains open.

\section{Conclusion}
\label{sec:conclusion}
 
This paper developed an input-space framework for local
controllability under state-dependent input constraints that strictly
exclude zero. The admissible balancing set \(\mathcal B(x_e)\)
distinguishes two qualitatively different regimes. An interior
balancing input permits a uniform local shift to a symmetric-control
problem, whereas the absence of any admissible balancing input places
the feasible-velocity set strictly on one side of a separating
hyperplane.

In the no-balancing regime, the separating covector defines a linear
barrier functional that increases at a uniform positive rate along
every admissible trajectory remaining near the reference state. This
provides a quantitative obstruction to STLC without requiring the
reference state to be an admissible controlled equilibrium, and
extends Brammer's separating geometry to a nonlinear,
state-dependent setting.

The planar two-cable CDPR illustrates both regimes. Interior
balancing, combined with appropriate bracket conditions, certifies
STLC. Under the restrictive bounds \([30,40]\) N, the same reference
state belongs to the no-balancing regime, for which interval
branch-and-bound certifies
\(\alpha'_{\rm cert}=0.5865936~\mathrm{m\,s^{-2}}\).
Simulations with extremal, adversarial, and randomly switched tensions
illustrate this certificate. Future work will address the
boundary-balancing regime and finite-time reachability through
admissible excursions.


\bibliographystyle{IEEEtran}
\bibliography{bibtac}

\begin{IEEEbiography}
[{\includegraphics[width=1in,height=1.25in,clip,keepaspectratio]
{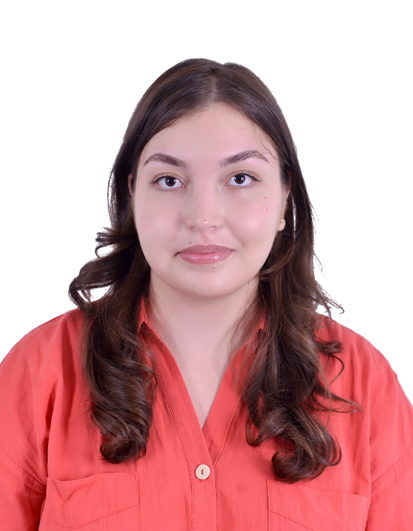}}]
{Amal Bouazza}
received the M.S. degree in applied mathematics and is pursuing a joint Ph.D. degree in nonlinear control with the International University of Rabat, Morocco, and the University of Lorraine, CNRS--CRAN, France. Her research interests include nonlinear and geometric control, constrained-system controllability, and underactuated cable-driven parallel robots using model- and AI-based approaches.
\end{IEEEbiography}
\begin{IEEEbiography}
[{\includegraphics[width=1in,height=1.25in,clip,keepaspectratio]
{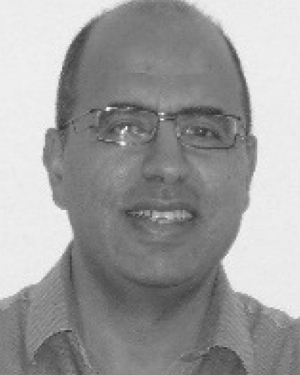}}]
{Mohamed Boutayeb}
received the engineering degree from EHTP, Morocco, in 1988, and the Ph.D. and HDR degrees in automatic control from the University of Lorraine, France, in 1992 and 2000, respectively. Since 2002, he has been a Full Professor with ICube, CNRS--University of Strasbourg, and CRAN, CNRS--University of Lorraine. He held delegations at CNRS and INRIA, has managed several ANR-DGA projects, and has authored over 200 publications. He has held research visits at MIT, the Alexander von Humboldt Foundation, and Columbia University. His research interests include state estimation and control of dynamical systems.
\end{IEEEbiography}
\begin{IEEEbiography}
[{\includegraphics[width=1in,height=1.25in,clip,keepaspectratio]
{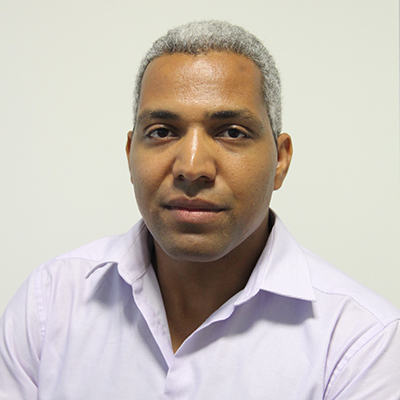}}]
{Mustapha Oudani}
received the joint Ph.D. degree in Computer Science and mathematics from Sidi Mohamed Ben Abdellah University, Fez, Morocco, and Le Havre Normandie University, France, and the HDR degree from the University of Angers, France. He is currently an Associate Professor with the School of Computer Science and Digital Engineering (ESIN), International University of Rabat, Morocco, and a member of TICLab.

His research interests include combinatorial optimization, operations research, mathematical control, computational mathematics, intelligent optimization, with applications to logistics, transportation, robotics, and sustainable systems.
\end{IEEEbiography}
\end{document}